%% file: implicit_timestepping.tex
\documentclass[reqno]{amsart}

\title{Optimal adaptive implicit time stepping}
\author{Michael Feischl and David Niederkofler}
\usepackage[a4paper, margin=2cm]{geometry}
\usepackage[hidelinks]{hyperref}
\usepackage{amsmath,amssymb}
\usepackage{amsthm}
\usepackage{amsfonts}
\usepackage{tcolorbox}
\usepackage{mathtools}
\usepackage{multirow}
\usepackage{caption}
\usepackage{subcaption}
\usepackage{tikz}
\usepackage{csvsimple} 
\usepackage{placeins}
\usepackage{algorithm}
\usepackage{algorithmic}
\usepackage{soul}
\usepackage{psfrag}

\usepackage[backend=biber, giveninits=true,sortcites]{biblatex} 
\input{mymacros}

\def\tend{t_{\rm end}}
\date{\today}

\begin{document}
\begin{abstract}
We revisit adaptive time stepping, one of the classical topics of numerical analysis and computational engineering. While widely used in
application and subject of many theoretical works, a complete understanding is still missing. Apart from special cases, there does not exist a complete theory that shows how to choose the time steps such that convergence towards the exact solution is guaranteed with the optimal convergence rate. In this work, we use recent advances in adaptive mesh refinement to propose an adaptive time stepping algorithm that is mathematically guaranteed to be optimal in the sense that it achieves the best possible convergence of the error with respect to the number of time steps, and it can be implemented using a time stepping scheme as a black box.
\end{abstract} 
\thanks{Funded by the Deutsche Forschungsgemeinschaft (DFG, German Research Foundation) -- Project-ID 258734477 -- SFB 1173, the Austrian Science Fund (FWF)
under the special research program Taming complexity in PDE systems (grant SFB F65) as well as project I6667-N. Funding was received also from the European Research Council (ERC) under the
European Union’s Horizon 2020
research and innovation programme (Grant agreement No. 101125225).}
\maketitle

\section{Introduction}

Adaptivity in time and space is ubiquitous in engineering and physics (see, e.g.,~\cite{krogh,physics,signal} for early results in adaptive time stepping and~\cite{amrgeo,amrwaves,burstedde1} current applications in geophysics). Particularly, adaptive step size control goes back to the 1890s, when C. Runge used computations with halved step sizes to find reliable digits in his calculations~\cite[Page~164]{hairer}. 
While spatial singularities in linear equations are often geometry induced (and sometimes can even be treated in an a~priori fashion), non-linear time-dependent problems can evolve temporal singularities that form spontaneously and have to be tackled with adaptive time stepping. To that end, adaptive step size control was adopted decades ago in engineering to steer time stepping algorithms such as Runge-Kutta methods (see, e.g.,~\cite{krogh} for early results). 
Those classical time stepping methods (see, e.g.,~\cite[Section~II.4]{hairer} and~\cite{soderlind}) base the choice of step size on local information only, e.g., by comparing methods of different orders or by using information of previous 
steps in order to compute a guess for the size of the following step. This has the advantage that only one pass through the time-interval is necessary, but can lead to the paradoxical situation that the step size is reduced in each step such that the final time is never reached. This makes it very hard to give mathematical convergence guarantees and error estimates.

One way to circumvent this problem is given in~\cite{adaptiveG}, which employs reliable  error estimators 
as well as an intricate sequence of refinement and coarsening of step sizes. The authors furthermore derive an a~priori computable minimal step size (depending on the prescribed error tolerance). By not permitting steps-sizes below the minimal step size, the algorithm ensures that the final time is always reached in finitely many steps and the prescribed error tolerance is satisfied. This method, however, is very problem specific and does not yield mathematically guaranteed convergence rates. We also mention~\cite{kehlet} which develops a rigorous error estimator for time dependent problems which also takes round-off errors into account.

The goal of this work is to propose a new adaptive time stepping algorithm that avoids the problems of existing adaptive methods and is provably optimal in the sense of adaptive mesh refinement, see, e.g.,~\cite{stevenson07,ckns,Carstensen_2014}. Roughly (see Section~\ref{sec:adaptiveAlg} below for a precise statement), this means that there holds the estimate
\begin{align*}
    {\rm error}\lesssim (\#{\rm time steps})^{-s}
\end{align*}
 for the maximal possible convergence rate $s>0$.
Besides a very recent optimal adaptive time stepping algorithm for the heat equation~\cite{generalqo}, no adaptive time stepping method with this property can be found in the literature.
Even plain convergence of the error to zero (without rates) has only been proven recently in~\cite{firstparabolic,adaptiveG,adaptiveDG,gregor} for space-time problems.

We are interested in the initial value problem
\begin{align}
\label{eq:ode}
\begin{split}
     \partial_t y(t)&=F(t,y(t)), \quad \text{for} \quad t \in [t_0, \tend]\\
    y(t_0)&=y_0,
\end{split}
\end{align}
 with right-hand side $F: \R \times \R^d \to \R^d$. This could come from an ODE or from a semi-discretized PDE (see Section~\ref{sec:heatexample} below for an example with the heat equation).

Given a time stepping method (e.g., Runge-Kutta) \fbox{Solve} as well as an error estimator \fbox{Estimate} that is an upper bound for the error, we propose the following algorithm: Create an initial set of time steps $t_0<t_1<\ldots<t_N=t_{\rm end}$ and repeat steps 1--3.
\begin{itemize}
\item[1)] Run \fbox{Solve} on the current set of time steps
\item[2)] Compute error estimator \fbox{Estimate}
\item[3)] Introduce new time steps $t=(t_i+t_{i+1})/2$ where error estimator is large, goto 1.
\end{itemize}
We refer the reader to Algorithm~\ref{alg:adaptive} below for a more precise description.
In the plots below, we see a demonstration of this adaptive method for the classical Van-der-Pol equation $\partial_t x = y$ and $\partial_t y = 5(1-x^2)y-x$ with initial values $x(0)=y(0)=1$. This problem served as the standard test bed for adaptive methods in the early days and thus is a nice introductory example. We plot the approximation of $t\mapsto x(t)$ over the set of time steps for different iterations $\ell$ of the adaptive algorithm. We clearly see that with an increase in the number of iterations the time steps get smaller in regions where the solution is more complex. The left plot zooms into the time interval $[0,10]$ and shows that the first oscillation is not picked up by the initial pass but the adaptive algorithm refines in such a way that the second iteration $\ell=1$ already resolves the oscillation partially. The plot on the right shows the time interval $[0,20]$ to demonstrate that all oscillations are resolved simultaneously.
\begin{center}
\includegraphics[width=0.49\textwidth, trim=80 30 80 40, clip]{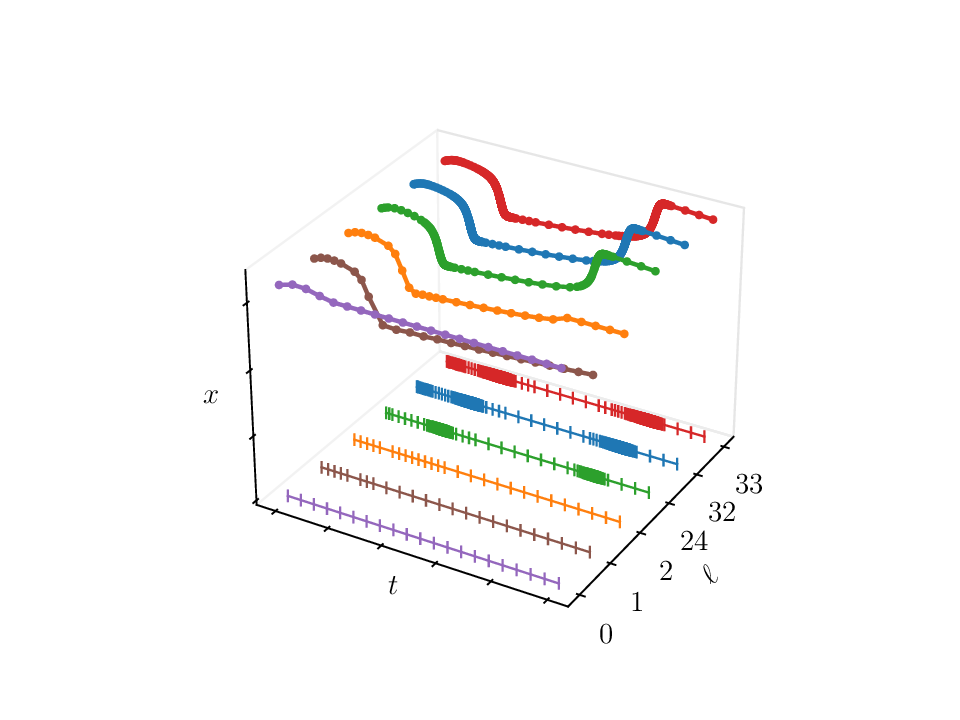}
\includegraphics[width=0.49\textwidth, trim=80 30 80 40, clip]{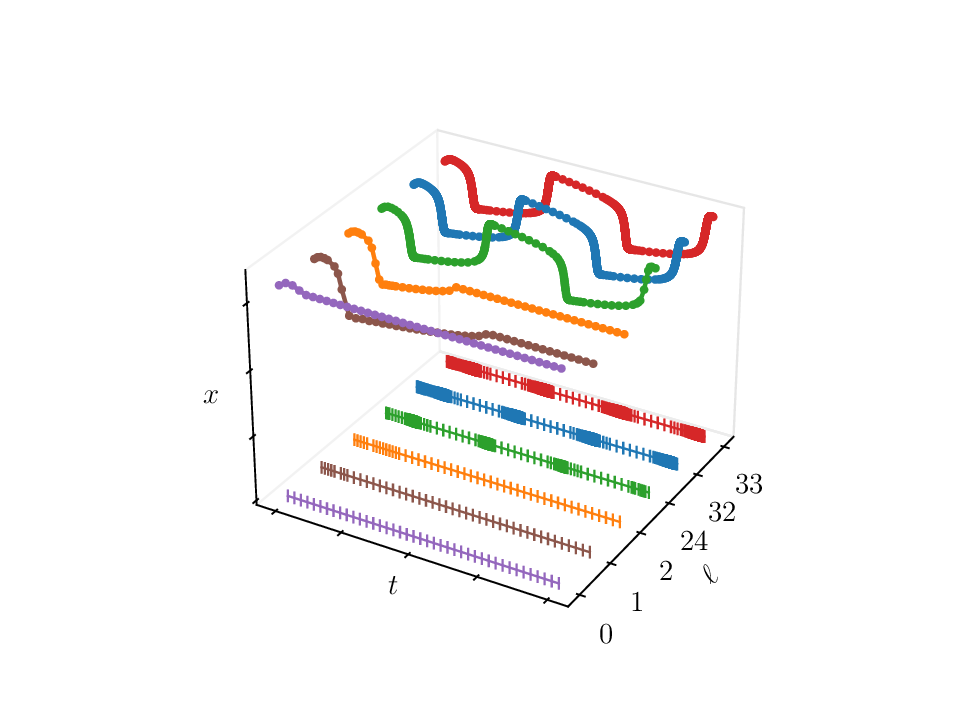}
\end{center}
Note that this adaptive procedure is very similar to adaptive mesh refinement in the stationary case by refining the time intervals and thus introducing finer step sizes and a similar idea can already be found in~\cite{estep}. This necessitates several passes through the time interval in order to obtain a good approximation of the exact solution. However, as we will see in the proof of Theorem~\ref{thm:optim} below, it is a consequence of linear convergence of the algorithm that the cost overhead of those multiple passes is dominated by the final computation on the finest set of time steps as sketched in the following plot:
\begin{center}
    \includegraphics[width=0.5\textwidth]{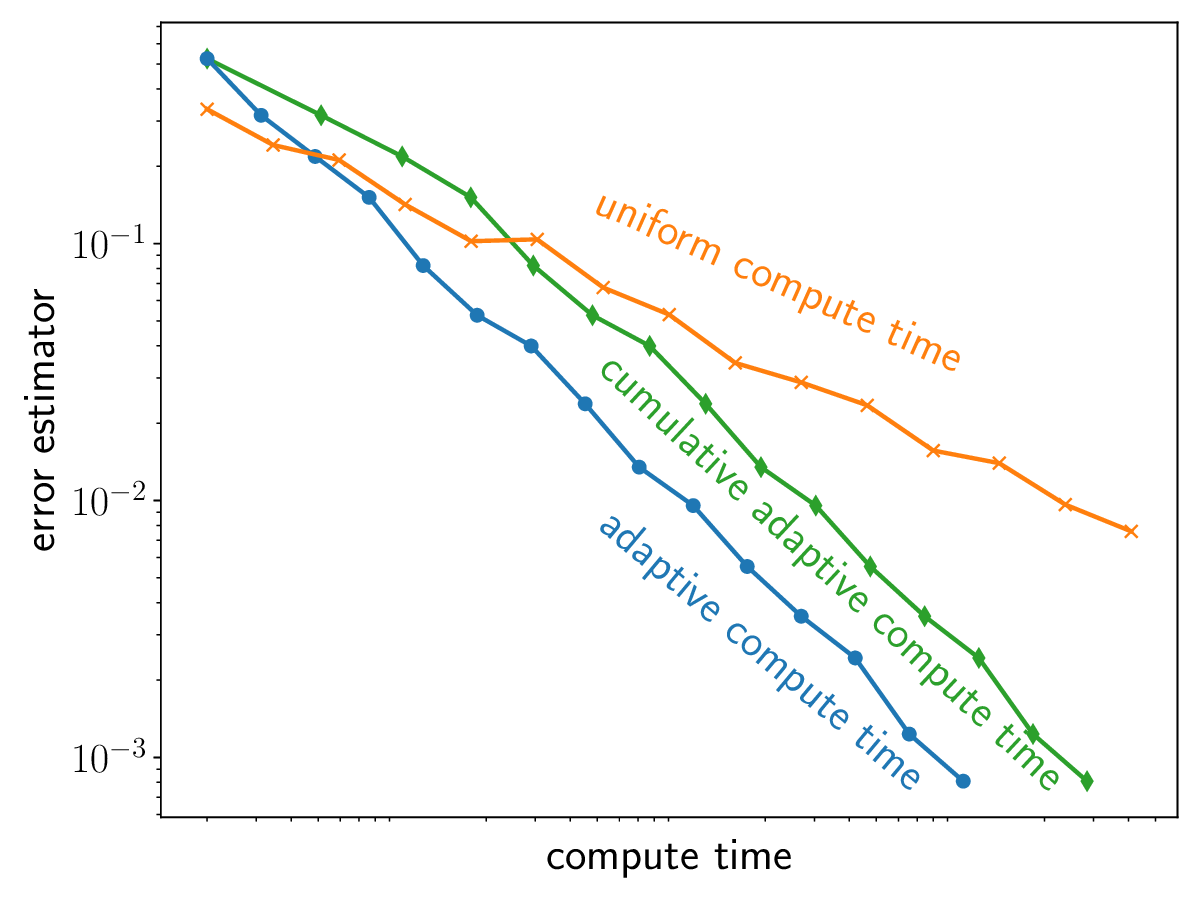}
\end{center}
We also refer to the real experiments in Figures~\ref{fig:time_ana}\&\ref{fig:time_ana_nonlin} below, where we track the compute-time for all passes through the domain and compare it with the uniform approach, which obviously requires only one pass. However, the adaptive method clearly outperforms the uniform method.
We also refer to~\cite{complexity} for a detailed discussion of provably optimal runtimes of adaptive algorithms of this kind in the stationary case, even if one accounts for an iterative (nonlinear) solver.

We note that the error estimator \fbox{Estimate} does not have to be (and in practice almost never is) a perfect proxy for the exact error between the approximate and the exact solution. This can be seen in the plots below. In the left plot, we see the exact solution of the Van-der-Pol equation ($\mu=2$) and a coarse approximation with step size $0.5$ computed with the Crank-Nicholson method. The right plot compares the exact error with the error estimator. We see that there is some correlation between error and estimator, but no perfect match.
\begin{center}
    \includegraphics[width=\textwidth]{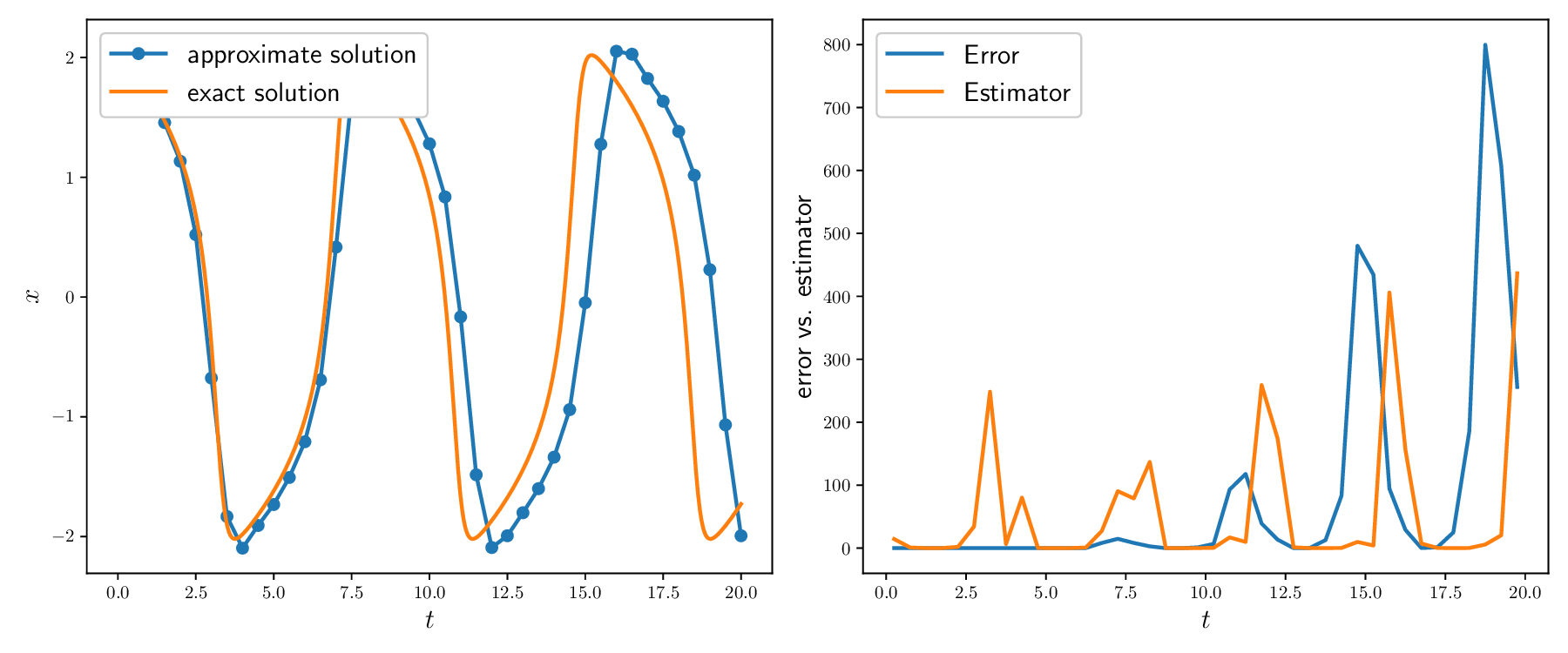} 
\end{center}
In fact, the error estimator should not be a perfect proxy for the error, as argued in Remark~\ref{rem:residual}. For the theoretical arguments below, we only require that the error estimator is an upper bound for the error and satisfies some other, more subtle properties, discussed in Section~\ref{sec:proof} below.

\subsection{Historical overview}
The classical adaptive step size methods select the step size of the next time step based on local information. A popular variant of these methods are so-called \emph{embedded} Runge-Kutta schemes (see, e.g.,~\cite[Chapter II.4]{hairer} for details). Those schemes are of the form $y_1 = y_0 + \tau(k_1b_1+\ldots +k_sb_s)$ and provide accuracy on the order $\tau^p$ if the exact solution is sufficiently smooth (here $\tau>0$ is the maximal time step size). A different linear combination of the same stages $k_1,\ldots,k_s$ gives another approximation $\widehat y_1 = y_0 + \tau(k_1\widehat b_1+\ldots +k_s\widehat b_s)$, which is then typically of order $\widehat p = p + 1$. Since the expensive part of any Runge-Kutta scheme is the evaluation of the function to obtain the stages $k_1,\ldots, k_s$, embedded schemes allow two approximations to be computed with minimal overhead. The difference of the two approximations is then used to estimate the error of the solution. This is based on the idea that
\begin{align}\label{eq:stepest}
|y(t_1)- y_1| \approx |\widehat y_1-y_1|
\end{align}
and can be used to select the next step size, e.g., one would decrease the step size if the error estimate is too large.  In adaptive mesh refinement, the assumption~\eqref{eq:stepest} is often called the \emph{Saturation Assumption}. However, unless one poses very severe regularity assumptions on the exact solution $y$, the heuristic~\eqref{eq:stepest} can not be guaranteed and is often violated even in practical applications.    This is the reason why the convergence to the exact solution and convergence rates of these adaptive methods are not well understood to this day. 

A completely different approach to adaptivity saw light in the 1980s and 1990s as I.~M.~Babuska initiated a wave of mathematical research by developing the first a~posteriori error estimators~\cite{adaptivefirst} for stationary PDEs. This opened the new field of mesh refinement in numerical analysis and led to tremendous research activity, we only mention the books~\cite{ao00,verf,repinalleinzuhaus}. D\"orfler~\cite{d1996} gave a first rigorous strategy for mesh refinement based on error estimators in 2D more than a decade later and after another five years Morin, Nochetto, and Siebert~\cite{mns} presented the first unconditional convergence proof.  We refer to~\cite{BECKER20241} for an overview on a posteriori error estimation and adaptivity.

Optimal convergence rates (see Section~\ref{sec:adaptiveAlg} below for a precise definition) of those adaptive algorithms were first shown by Binev, Dahmen, and DeVore~\cite{bdd} in 2004 for adaptive FEM of the Poisson model
problem. Building on that, Stevenson~\cite{stevenson07} and Cascon, Kreuzer, Nochetto,  and Siebert~\cite{ckns} simplified the algorithms and extended the problem class. 
These early works established what is now called the standard adaptive loop:
\begin{align}\label{eq:adaptiveloop}
 \fbox{Solve}\longrightarrow\fbox{Estimate}\longrightarrow\fbox{Refine}
\end{align}
However, all this theory crucially relies on the fact that the considered problems are symmetric with respect to the differential operators, e.g.,
\begin{align*}
\int \Delta u v\,dx = -\int \nabla u \nabla v\,dx=\int u \Delta v\,dx
\end{align*}
for any domain with zero boundary conditions. This is no longer the case in the time dependent case, where $\partial_t u$ spoils the party ($\int \partial_t u v\,dt = -\int u \partial_t v\,dt$). Thus, one can not directly transfer the results and arguments to time stepping. It took until 2019 for the first adaptive methods that are provably optimal for non-elliptic problems to appear (see, e.g.,~\cite{stokesopt} for the first optimal adaptive methods for the Stokes problem). Finally, in 2022, the work~\cite{generalqo} introduced a new proof technique to also tackle time dependent problems and builds the theoretical basis for this work. The recent work~\cite{nochettoacta} gives a great overview of the current state of the art of adaptive mesh refinement.

\subsection{Notation}
The space of continuous functions on an interval $(a,b)$ is denoted by $C^0(a,b)$.
We denote the $L^2$-norm on a domain $\Omega$ by $\norm{\cdot}{L^2(\Omega)}$. The Sobolev space $H^1_0(\Omega)$ is defined as the space of all $L^2(\Omega)$ functions with weak derivative in $L^2(\Omega)$ and zero boundary conditions. We denote the dual space of $H^1_0(\Omega)$ by $H^{-1}(\Omega)$.  The notation $\partial_t u$ denotes the time derivative of a function $u$ and $\nabla_y F(t,y)$ denotes the gradient of $F$ with respect to its second argument.  The notation $\#\TT$ denotes the number of elements in a set $\TT$.
\section{Implicit Discretization}
\label{sec:impl_adaptivity}
We consider the initial-value problem \eqref{eq:ode}, and assume that $F$ is Lipschitz-continuous, i.e.
\begin{align}
\label{eq:Lip}
    |F(t,y_1)-F(t,y_2)| \leq L_1 |y_1-y_2| \quad \text{for all } y_1,y_2 \in \R^d,\, t \in [t_0,\tend].
\end{align}
Furthermore, we suppose that $F$ is continuously differentiable with Lipschitz-continuous Jacobian, i.e.
\begin{align}
\label{eq:JacLip}
    |\nabla F(t_1,y_1)-\nabla F(t_2,y_2)| \leq L_2\Big(|t_1-t_2|+|y_1-y_2| \Big) \quad \text{for all } y_1,y_2 \in \R^d,\, t_1, t_2 \in [t_0,\tend].
\end{align}
%Lastly we suppose that $F$ is piecewise affine in $t$, i.e. there exists a discretization $t_0< t_1 <\ldots < t_n=\tend$ denoted by $\TT_0= \{[t_0,t_1], \ldots, [t_{n-1}, t_n] \}$ of $[t_0,\tend]$, $n=\#\TT_0$, such that
%\begin{align}
%\label{eq:affine}
%    F(t,y)=f_i(y)+\frac{f_{i+1}(y)-f_{i}(y)}{t_{i+1}-t_i}(t-t_i), \quad t \in [t_i, t_{i+1}], i \in {0, \ldots n-1},
%\end{align}
%and functions $f_0, \ldots, f_n$ with the required properties.
We suppose there is an initial discretization set of time steps $t_0< t_1 <\ldots < t_n=\tend$ which induces a time-mesh (a set of time intervals) $\TT_0= \{[t_0,t_1], \ldots, [t_{n-1}, t_n] \}$ of $[t_0,\tend]$, $n=\#\TT_0$. In the following, we discretize the ODE~\eqref{eq:ode} on a mesh $\TT$ which is a refinement of $\TT_0$, i.e., it has additional time steps. We will use the spline spaces
\begin{align*}
    \PP^p(\TT)&:=\set{v\in L^2(t_0,\tend)}{v|_{T}\text{ is a polynomial of degree }\leq p,\, \text{for all }T\in\TT},\\
    \SS^p(\TT)&:=\PP^p(\TT)\cap C^0(t_0,\tend).
\end{align*}
The space $\PP^p(\TT)$ contains piecewise polynomials of degree $p$ with potential jumps over the time-steps, while functions in $\SS^p(\TT)$ are continuous over the whole time interval. 
In the following, we consider continuous Petrov-Galerkin finite-element-in-time methods (see, e.g.,~\cite[Chapter~70]{ErnGuermond}, for an overview) of the following form: Find $y_\TT \in S^p(\TT)$ such that $y_\TT(t_0) = y_0$ and
\begin{align}
\label{eq:gen_disc}
\begin{split}
    \int_{t_0}^{\tend}\Big(\partial_t y_\TT(t)-F(t,y_\TT(t)) \Big)\cdot v(t) dt &=0 \quad\text{for all } v\in \PP^{p-1}(\TT).
\end{split}
\end{align}
This method is indeed a time stepping method as one can test~\eqref{eq:gen_disc} with function $v\in \PP^{p-1}(\TT)$ that are supported only on one time interval $[t_i,t_{i+1}]$. Then, $y_\TT|_{[t_i,t_{i+1}]}$ can be computed using only the data $y_\TT(t_i)$.
We note that for linear ODEs, the case $p=1$ corresponds to the Crank-Nicolson method, while the case $p=2$ corresponds to the three stage Lobatto 3A method.

In the following, we want to estimate the approximation error and use that information to adjust the time step size iteratively.
To that end, we use a residual-based element-wise error estimator similar to \cite{generalqo} that measures the error on each time interval separately, i.e.
\begin{align}
\label{eq:estimator}
\begin{split}
    \eta_{\TT}^2&= \sum_{T \in \TT} \eta_\TT(T)^2,\\
    \eta_\TT(T)&=|T|\norm{\partial_t F(\cdot,y_{{\TT}})+\nabla_yF(\cdot, y_{{\TT}}) \partial_t y_{{\TT}}-\partial^2_t y_\TT}{L^2(T)}.
\end{split}
\end{align}
We will show in Section~\ref{sec:proof} below that this estimator is reliable in the sense that it is an upper bound for the error. Moreover, for many problems, this error estimator is easily computable without much overhead compared to the actual computation of $y_\TT$.
 
\begin{remark}\label{rem:residual}
    We note that, different from diffusion-dominated stationary problems, the optimal error estimator is not the error itself.
    The fact that one would, in theory, like to steer the adaptive algorithm with the actual error has guided the development of many error estimators like $h-h/2$, two-level, or enrichment based error estimators in finite elements. However, for time dependent problems, this cannot work and some form of residual type error estimation seems to be necessary. To illustrate this idea, assume that the numerical approximation $y_\TT$ satisfies~\eqref{eq:diffeq} with $F(t,y)=y$ exactly except for one element $T_0$ with appears early in the time interval. A residual type error estimator would correctly identify the element $T_0$ as the one which needs to be refined since $\partial_t y_\TT-F(y_\TT)\neq 0$ on $T_0$. However, the approximation error is amplified exponentially, and a corresponding error estimator would always mark elements at the very end of the time interval for refinement, never identifying the root cause of the problem.
\end{remark}

\subsection{Example: Semi-discretized heat-equation}\label{sec:heatexample}
The linear heat equation
\begin{align*}
    \partial_t u - \Delta u &= f\quad\text{in }[t_0,\tend]\times \Omega,\\
    u&=0\quad\text{on }[t_0,\tend]\times \partial\Omega,\\
    u(0)&=u_0\quad \text{on }\Omega
\end{align*}
is usually discretized in space via some finite element method (FEM), i.e.,
\begin{align*}
    u(t) \approx  \sum_{i=1}^d \phi_i y_\TT(t)_i,
\end{align*}
where $\{\phi_1,\ldots,\phi_d\}$ are a basis of the underlying finite element space $\XX_h\subseteq H^1_0(\Omega)$.
This results in a problem of the form~\eqref{eq:ode} with
\begin{align*}
    F(t,y):= M^{-1}(Ay + f),
\end{align*}
where $A,M\in \R^{d\times d}$ denote the stiffness and the mass matrix of the FEM. Since the error estimator~\eqref{eq:estimator} effectively measures the residual 
of the problem at hand, it is useful to think about the natural norm for it. While all norms are equivalent on $\R^d$, the correct norm will result in consistent results for different space mesh sizes $h>0$ (which correspond to different $d\in\N$). For the heat equation, we have $\partial_t u-\Delta u-f \in L^2([t_0,\tend],H^{-1}(\Omega))$, where $H^{-1}(\Omega)$ denotes the dual space of $H^1_0(\Omega)$. Thus, for the discretization, it makes sense to equip $\R^d$
with the discrete analog of the $H^{-1}(\Omega)$-norm, which is given by $|y|_{MA^{-1}M}^2:= A^{-1}My\cdot My$ for all $y\in \R^d$.
With this norm, the elementwise contribution of the error estimator~\eqref{eq:estimator} read
\begin{align*}
    \eta_\TT(T)^2=|T|^2\int_T (A^{-1}M\partial_t^2 y_{{\TT}} -\partial_t y_{{\TT}}-A^{-1}\partial_t f)\cdot(M\partial_t^2 y_{{\TT}} -A\partial_t y_{{\TT}}-\partial_t f)\,dt.
\end{align*}
If one wants to avoid inverting $A$ to compute the error estimator, one can use the scaling behavior of $\norm{A^{-1}M}{2}\lesssim 1$ and ${\rm cond}(M)\simeq 1$ on finite element spaces $\XX_h$ based on uniform space meshes and use the norm $|y|_{MM}^2:=My\cdot My$ for $y\in \R^d$. This results in the error estimator
\begin{align*}
     \eta_\TT(T)^2=h^{2-d}|T|^2\int_T |M\partial_t^2 y_{{\TT}} -A\partial_t y_{{\TT}}- \partial_t f|^2\,dt.
\end{align*}

\subsection{Practical implementation}\label{sec:quad}
For nonlinear $F(\cdot)$, we usually have to approximate the integral in~\eqref{eq:gen_disc} by a quadrature rule. The simplest variant is the trapezoidal rule, which approximates the integral by evaluating the integrand at the endpoints of the time interval, i.e., for $p=1$ and $T=[t_0,t_1]\in\TT$, we have
\begin{align*}
    \int_T\Big(\partial_t y_\TT(t)-F(t,y_\TT(t)) \Big)\cdot v(t) dt\approx \frac{|T|}{2}\sum_{i=0}^1\Big(\partial_t y_\TT(t_i)-F(t_i,y_\TT(t_i)) \Big)\cdot v(t_i).
\end{align*}
A more accurate approximation is given by the Simpson rule, i.e., for $p=2$
\begin{align*}
    \int_T\Big(&\partial_t y_\TT(t)-F(t,y_\TT(t)) \Big)\cdot v(t) dt\\
    &\approx \frac{|T|}{6}\sum_{i=0}^1\Big(\partial_t y_\TT(t_i)-F(t_i,y_\TT(t_i)) \Big)\cdot v(t_i)
    +\frac{2|T|}{3}\Big(\partial_t y_\TT(t_{1/2})-F(t_{1/2},y_\TT(t_{1/2})) \Big)\cdot v(t_{1/2}),
\end{align*}
where $t_{1/2}=(t_0+t_1)/2$ is the midpoint of the time interval $T$. As mentioned above, these correspond to the Crank-Nicolson and the three stage Lobatto 3A methods. 

By choosing different quadrature rules, we can recover many Runge-Kutta methods. One useful example is Radau quadrature, which recovers the well-known Radau IIA 
methods~\cite{hairer2,hairer1999}, see Section~\ref{sec:VDP} below for an example.

It is important, that we choose the polynomial degree $p$ corresponding to the quadrature rule, in order to keep the equation~\eqref{eq:gen_disc} well-defined.
For example, a quadrature rule with $m\in\N$ quadrature points per element $T$ allows us to choose $p=m-1$ (otherwise we could construct non-zero $y_\TT$ that vanishes on all quadrature points).
Thus, we expect the optimal convergence rate in the $H^1$-norm to be of order $m-1$.
Thus, we expect the Crank-Nicolson method to show order one, the three stage Lobatto 3A method to show order two, and the typical Radau IIA methods Radau 5, Radau 9, and Radau 13 to show the orders three, five, and seven respectively (note that Radau quadrature with $m$ points is of order $2m-2$ and thus the global error of order $2m-3$).

In this case we may test with functions $v\in \PP^{m-2}(T)$. As the ansatz function $y_\TT|_T\in\PP^{p}(T)$ has $p-1$ degrees of freedom (one is fixed from the previous time step), we can choose $v$ such that it vanishes at all but one of the $m-1$ remaining quadrature points.

\subsection{The adaptive Algorithm}\label{sec:adaptiveAlg}
We note that for uniform time steps, convergence of the method above is well-known and can be of arbitrarily high order as long as the exact solution $y$ is sufficiently smooth, e.g., one can expect convergence of order $p$ if $y\in C^{p+1}([t_0,\tend])$. However, we are interested in adaptive choice of time steps and still prove convergence rates, ideally without many assumptions on the regularity of the exact solution.
We want to apply adaptive mesh refinement as introduced and analyzed in \cite{Carstensen_2014} and \cite{generalqo}. This is realized by the following algorithm.
\begin{algorithm}[H]
    \caption{Adaptive Algorithm}
    \label{alg:adaptive}
    \begin{algorithmic}[1]
       \item[] \textbf{Input:} Initial time mesh $\TT_0=\{[t_0,t_1],\ldots,[t_{N-1},\tend]\}$, parameter $0<\theta<1$. \\
        For $\ell=0,1,2,\ldots$ do:\\
        \STATE Compute $y_{\TT_\ell}$ from \eqref{eq:gen_disc}.\\
        \STATE Compute error estimates $\eta_{\TT_\ell}(T)$ for all $T \in \TT_\ell$.\\
        \STATE Find a set $\mathcal{M}_\ell$ of minimal cardinality such that
        \begin{align*}
            \sum_{T \in \mathcal{M}_\ell} \eta_{\TT_\ell}(T)^2 \geq \theta \sum_{T \in \TT_\ell} \eta_{\TT_\ell}(T)^2.
        \end{align*}\\
        \STATE Bisect all $T \in \mathcal{M}_\ell$ to obtain a new mesh $\TT_{\ell+1}$. \label{step:bisect}
    \end{algorithmic}
\end{algorithm}
As discussed in the introduction, we are interested in the best possible error vs. work rate. Mathematically, this can be stated as
\begin{align*}
    \eta_\TT\leq C (\#\TT)^{-s}
\end{align*}
for the largest possible rate $s>0$. We will show below, that the error estimator is an upper bound for the error. This means that any convergence rate for the error estimator immediately implies the same rate for the error, i.e.,
\begin{align*}
\norm{y-y_\TT}{H^1([t_0,\tend])}\leq C (\#\TT)^{-s}.
\end{align*}

To characterize the best possible rate $s>0$, we need to specify the set $\T$ of all possible time discretizations $\TT$, that can be reached from the initial discretization $\TT_0$ by bisection as stated in Step~\ref{step:bisect} of Algorithm~\ref{alg:adaptive}.
These are all discretizations $\TT$ with time steps $t_\TT$ that can constructed by bisecting previous time intervals $[t_i,t_{i+1}]$.
Given the initial discretization $\TT_0=\{[t_0,t_1],\ldots,[t_{n-1},t_n]\}$, all possible time steps $t_\TT$ are of the form
\begin{align*}
    t_\TT= t_i + k2^{-m} (t_{i+1}-t_i)
\end{align*}
for some $i=0,\ldots,n-1$, $m\in\N$ and $0\leq k\leq 2^m$.

If in this set $\T$ of possible discretizations, there exists a sequence $\TT_\ell^{\rm best}$, $\ell\in\N$ with $\#\TT_\ell^{\rm best}\to \infty$ and
\begin{align*}
    \eta_{\TT_\ell^{\rm best}}\leq C_{\rm best} (\#\TT_\ell^{\rm best})^{-s}\quad\text{for all }\ell\in\N,
\end{align*}
we say that the convergence rate $s>0$ is possible. This statement is equivalent to requiring that $C_{\rm best}$ is finite, i.e.,
\begin{subequations}
\label{eq:rate_optim}
    \begin{align}
        C_{\rm best} := \sup_{N \in \N } \inf_{\substack{\TT \in \mathbb{T}\\ \#\TT \leq N}} N^s \eta_{\TT} < \infty.
    \end{align}
We say that the adaptive algorithm is optimal for any possible rate $s>0$, if it computes approximations $y_{\TT_\ell}$ on time meshes $\TT_\ell$, $\ell\in\N$, with
\begin{align}
    \eta_{\TT_\ell}\leq C_{\rm best}C_{\rm opt} (\#\TT_\ell)^{-s}\quad\text{for all }\ell\in\N
\end{align}
\end{subequations}
for some constant $C_{\rm opt}<\infty$ that does not depend on $\ell$. This behavior is also sketched in the following plot
\begin{center}
    \includegraphics[width=0.5\textwidth]{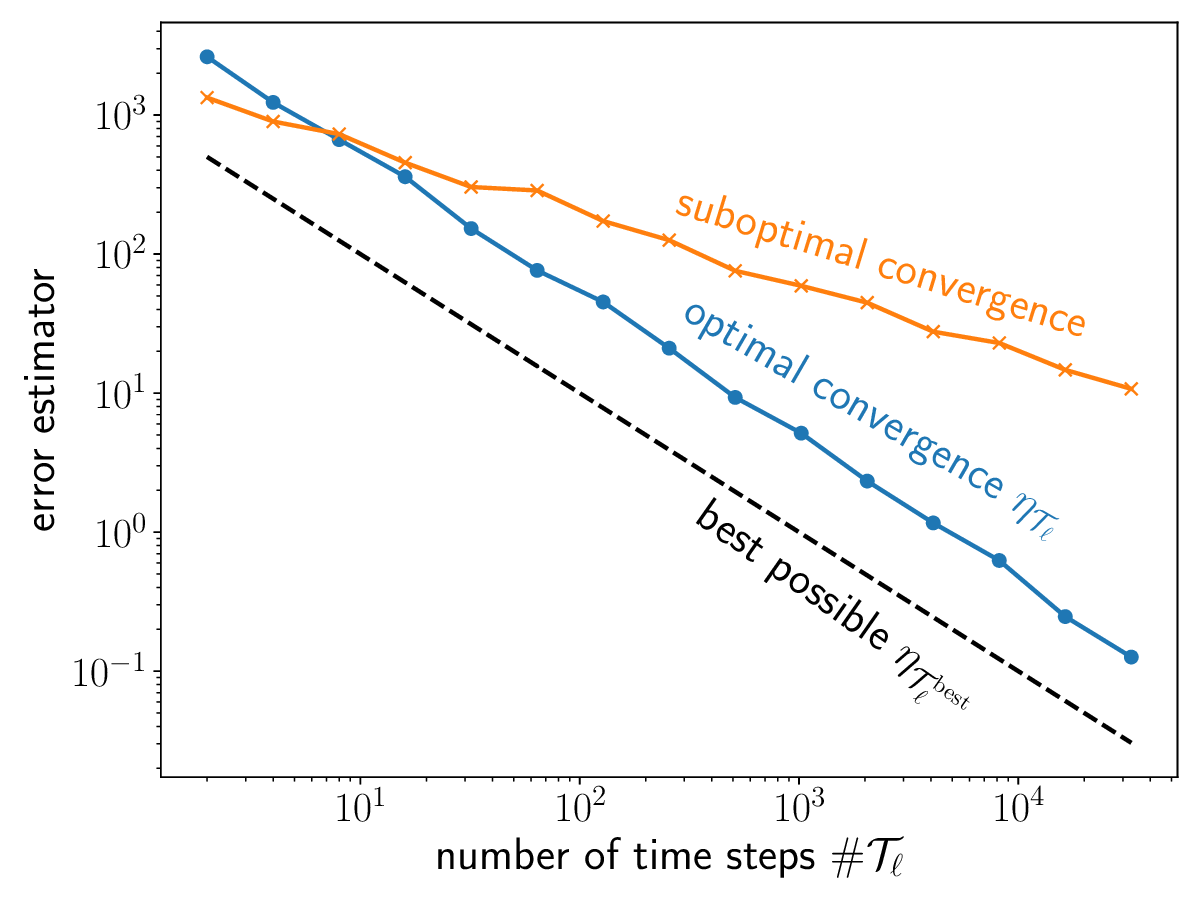}
\end{center}

With this notion of optimality, we will prove the following theorem.
\begin{theorem}
\label{thm:optim}
If the right-hand side $F$ from \eqref{eq:ode} satisfies \eqref{eq:Lip}, the error estimator~\eqref{eq:estimator} is reliable in the sense
\begin{align*}
    \norm{y-y_{\TT}}{H^1([t_0,\tend])}\leq C_{\rm rel}\eta_{\TT}\quad\text{for all admissible meshes }\TT\in\T,
\end{align*}
where $C_{\rm rel}$ is independent of $\TT$. If $F$ satisfies additionally~\eqref{eq:JacLip} and if $\TT_0$ is sufficiently fine, then there exists $0<\theta_\star<1$ such that Algorithm \ref{alg:adaptive} satisfies rate-optimality~\eqref{eq:rate_optim} for all $0<\theta<\theta_\star$. Moreover, the error and the error estimator converge with any rate $s>0$ such that $C_{\rm best}$ from~\eqref{eq:rate_optim} is finite, i.e.,
\begin{align}\label{eq:optimal}
    \norm{y-y_{\TT_\ell}}{H^1([t_0,\tend])}\leq C_{\rm rel}\eta_{\TT_\ell}\leq C_{\rm best}C_{\rm opt} C_{\rm rel} (\#\TT_\ell)^{-s}\quad\text{for all }\ell\in\N.
\end{align}
Finally, if we assume that $y_{\TT_\ell}$ can be computed in $C_{\rm time}\#\TT_\ell$ time for all $\ell\in\N$ (i.e., constant cost per time step), there also holds
\begin{align}\label{eq:runtime}
 \norm{y-y_{\TT_\ell}}{H^1([t_0,\tend])}\leq C_{\rm rel}\eta_{\TT_\ell}\leq  C_{\rm rel}C_{\rm time}^s\frac{C_{\rm best}C_{\rm opt}C_{\rm lin}}{(1-q^{1/s})^s} (\text{total runtime of Algorithm~\ref{alg:adaptive} up to step }\ell)^{-s}.
\end{align}
Note that the total runtime ignores bookkeeping and marking, which is usually orders of magnitude faster (see also~\cite{carlmarks} for a time-optimal marking algorithm). 
\end{theorem}
For a more detailed discussion of optimal runtimes~\eqref{eq:runtime}, including iterative solver time for nonlinear stationary problems, we refer to~\cite{complexity,dirknonlin}.   

\subsection{Estimator modification based on confidence}\label{sec:modification}
Particularly for challenging stiff problems, we observe that the vanilla form of the estimator results in many unnecessary time steps as \emph{spurious} near singularities appear in the pre-asymptotic regime of the computation. This is particularly pronounced for long time intervals. To remedy this problem, we found that a simple modification of the estimator based on the confidence we have on estimated error can significantly improve the performance. We achieve this by defining the locally equivalent estimator
\begin{align*}
\widetilde \eta_\TT(T_i)^2:= \frac{\eta_\TT(T_i)^2}{1+\sum_{j=0}^i\eta_\TT(T_j)^2}\quad\text{for all }T_i\in\TT,
\end{align*}
where $T_1,T_2,\ldots$ are the ordered time intervals in $\TT$. The idea is that large estimator contributions early in the time interval decrease confidence in the later contributions. Thus, Algorithm~\ref{alg:adaptive} will focus on the earlier intervals with large estimator. We immediately see local equivalence  if we have an upper bound $\eta_\TT^2\leq C$ for all $\TT$, i.e.,
\begin{align*}
\widetilde \eta_\TT(T_i)^2\leq \eta_\TT(T_i)^2\leq (1+C)\widetilde \eta_\TT(T_i)^2
\end{align*}
Usually, the bound $C$ follows from the stability properties of the given problem. We note that we use $\widetilde \eta_\TT$ only for the marking step in Algorithm~\ref{alg:adaptive}. Local equivalence implies that all the theoretical results from Theorem~\ref{thm:optim} still hold, i.e., the adaptive algorithm is still optimal in the sense of~\eqref{eq:rate_optim} and~\eqref{eq:optimal} with respect to the original estimator $\eta_\TT$. For details, we refer to~\cite[Section~8]{Carstensen_2014}.

\section{Applications of adaptive time stepping} 
Before we discuss the proof of Theorem~\ref{thm:optim}, we want to showcase the usefulness of adaptive time stepping on a couple of classical problems.
While older adaptive time stepping methods as discussed in the introduction may show similar practical performance, we emphasize that the present methods are the first for which one can prove optimal convergence rates in the sense of Theorem~\ref{thm:optim}. Thus, the performance is mathematically guaranteed.
In the examples below, we compare uniform time steps $t_i= t_0+ \tau i$ for some step size $\tau:=1/\#\TT>0$ with the adaptive step size selection from Algorithm~\ref{alg:adaptive}. Since the exact solution is usually not available, the errors are computed with respect to the approximation on the finest set of time steps. As an error measure, we plot the $H^1([t_0,\tend])$-norm, which is the natural norm in which our error estimates are given. 
Note that this norm is an upper bound for the pointwise error $\max_{t\in [t_0,\tend]}|y(t)-y_\TT(t)|$.

\subsection{Linear Heat Equation}
We consider the heat equation with homogeneous Dirichlet boundary conditions as discussed in Section~\ref{sec:heatexample} on $\Omega:=[0,1]^2$ and for $f=0$.
We set $u_0=1$ and project it onto $\XX_h$ via the $L^2(\Omega)$ projection. Furthermore, we use the error estimator derived in Section~\ref{sec:heatexample}. Below, we present the results of several experiments with different order discretizations.
In all the experiments, we see a clear advantage of the adaptive algorithm over the uniform algorithm. This is due to the startup singularity of the exact solution
at $t=0$ induced by non-matching boundary conditions $u_0|_{\partial\Omega}\neq 0$.
We also include plots that compare the absolute compute time for adaptive and uniform approach, and observe that the adaptive algorithm is superior even for coarse approximations. The difference between uniform and adaptive approach is more pronounced for higher order discretizations. Similar experiments for $p=1$ can be found in~\cite{generalqo}.
\begin{figure}[h]
    \centering
    
    \begin{subfigure}{0.45 \textwidth}
        \centering
        \includegraphics[width=\textwidth, trim=25 0 30 30, clip]{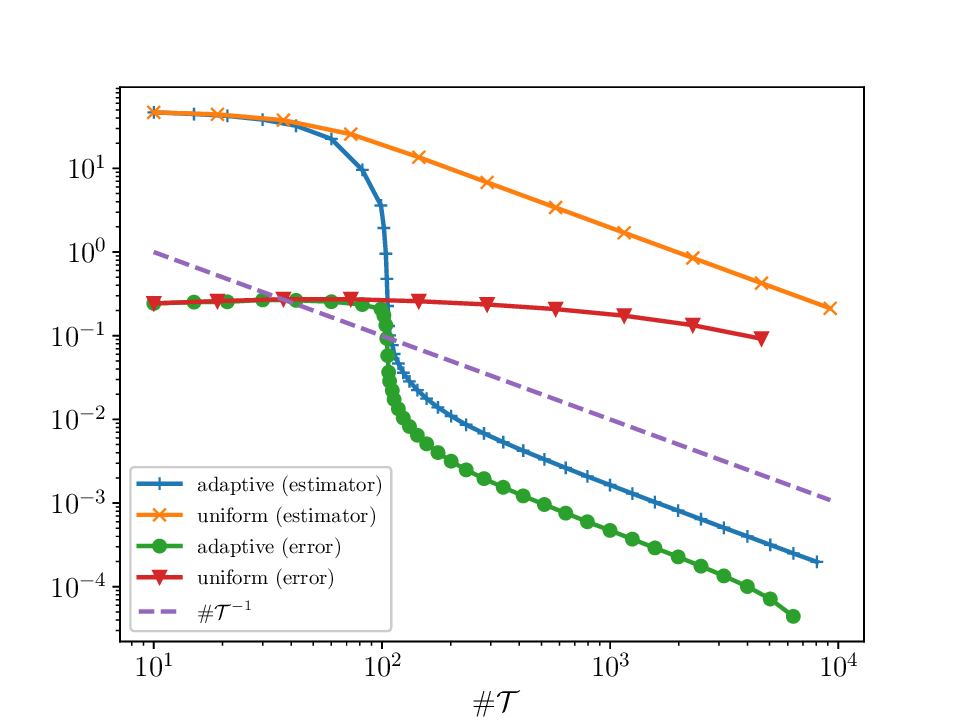}
        \caption{CN discretization, $p=1$}
    \end{subfigure}
    \begin{subfigure}{0.45 \textwidth}
        \centering
         \includegraphics[width=\textwidth, trim=25 0 30 30, clip]{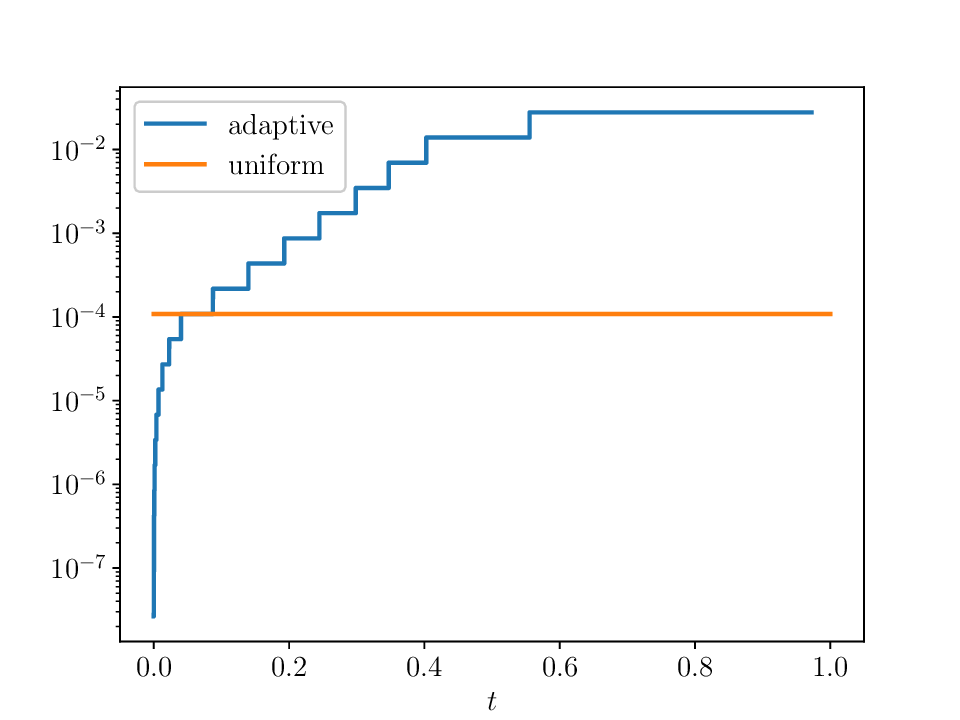}
          \caption{CN discretization, $p=1$}
    \end{subfigure}\\
     \begin{subfigure}{0.45 \textwidth}
        \centering
         \includegraphics[width=\textwidth, trim=25 0 30 30, clip]{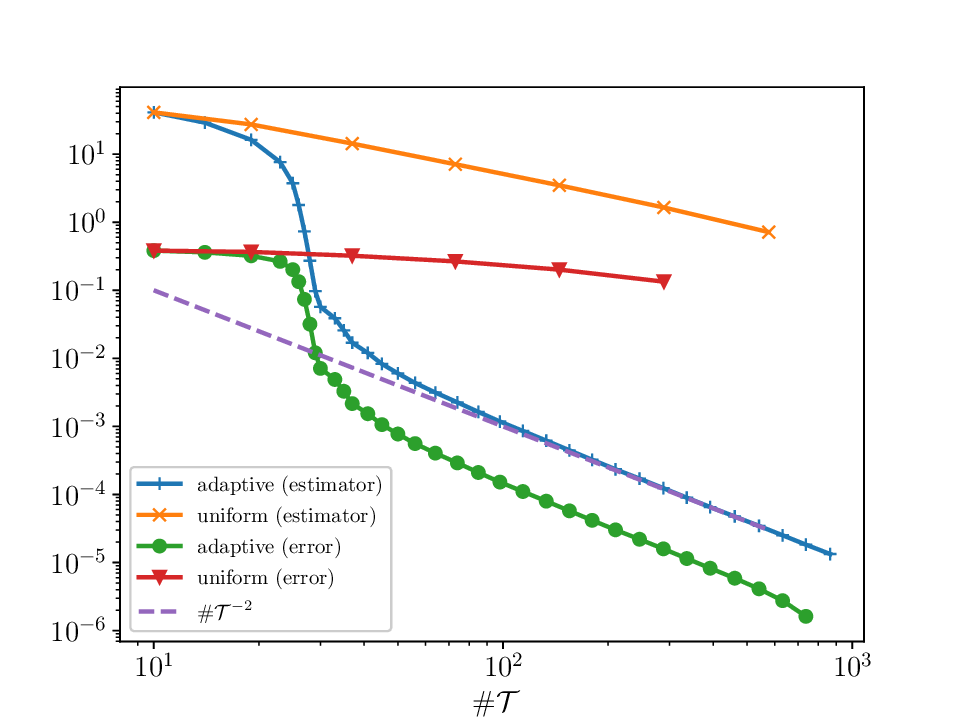}
         \caption{Lobatto discretization, $p=2$}
    \end{subfigure}
    \begin{subfigure}{0.45 \textwidth}
        \centering
         \includegraphics[width=\textwidth, trim=25 0 30 30, clip]{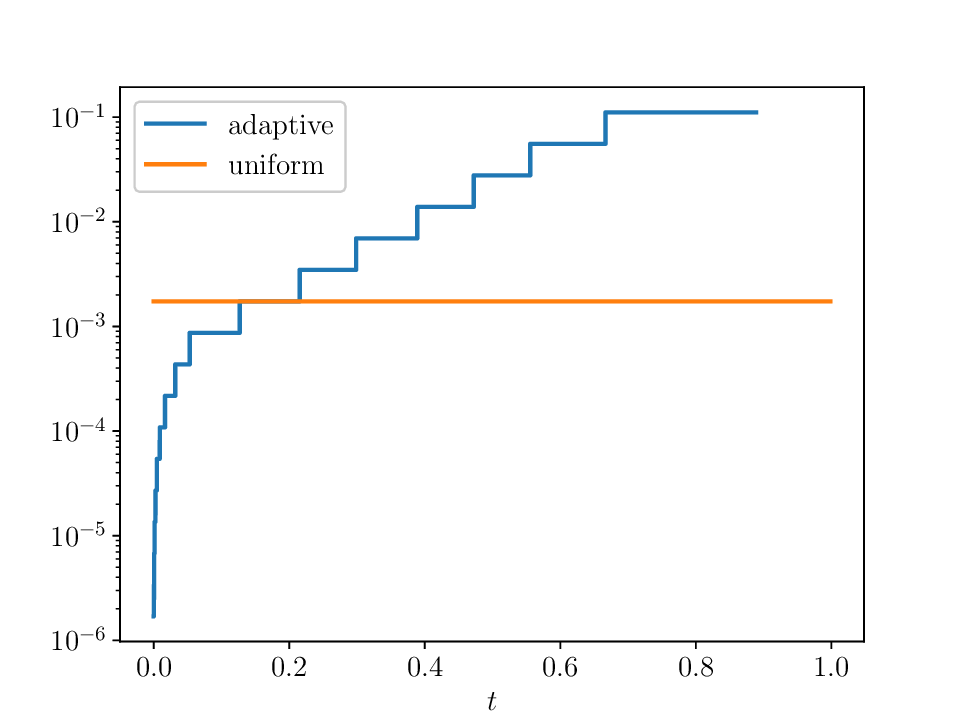}
         \caption{Lobatto discretization, $p=2$}
    \end{subfigure}
    \caption{(Linear heat equation) Left: We see the $H^1_0$-error $\norm{\partial_t(u-u_\TT)}{L^2((t_0,\tend))}$ and the estimator $\eta_\TT$ over number of time-intervals $\#\TT$. The space discretization uses a mesh-size of $\approx 1/50$ for $p=1$ and of $\approx 1/20$ for $p=2$. Right: We see the adaptive step sizes compared with the uniform step size of the finest computations. The parameter for the adaptive algorithm is $\theta=0.5$.  }
    
    \label{fig:heat_lin}
\end{figure}

\begin{figure}[h]

    \centering
    \begin{subfigure}{0.45 \textwidth}
        \centering
        \includegraphics[width=\textwidth, trim=25 0 30 30, clip]{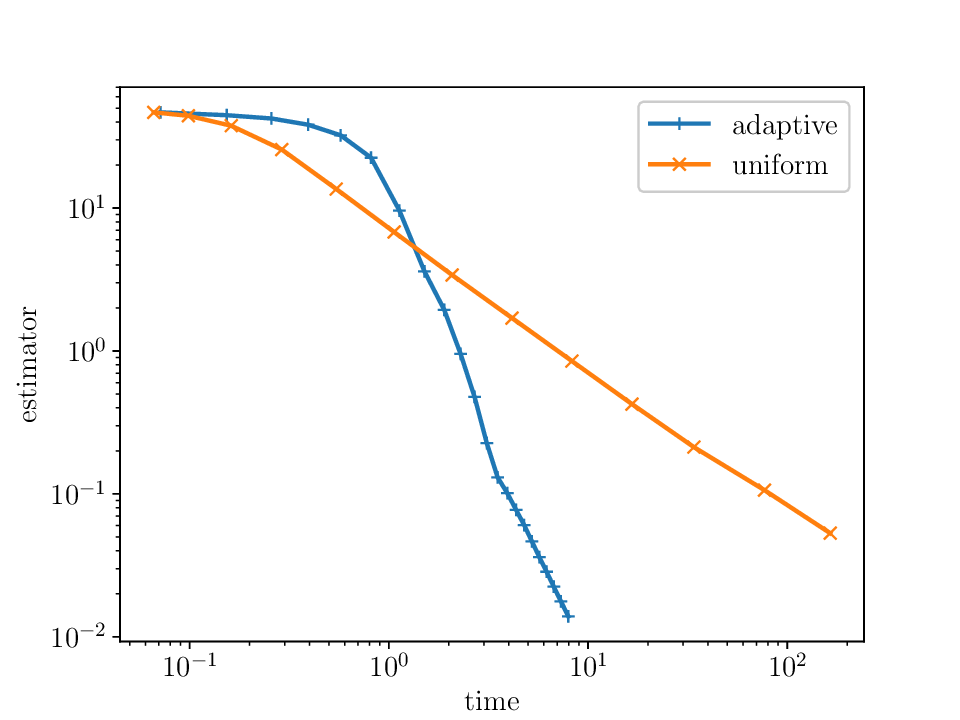}
        \caption{CN discretization, $p=1$}
    \end{subfigure}
    \begin{subfigure}{0.45 \textwidth}
        \centering
         \includegraphics[width=\textwidth, trim=25 0 30 30, clip]{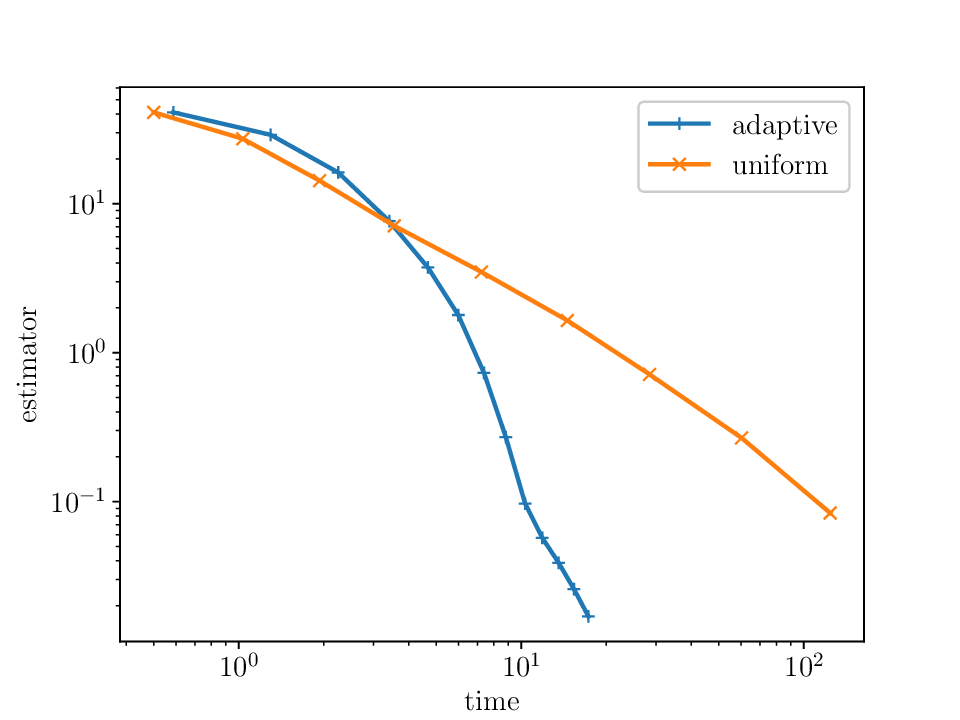}
          \caption{Lobatto discretization, $p=2$}
    \end{subfigure}\\
    \caption{(Linear heat equation) We plot the error estimator $\eta_\TT$ over the total compute time. For the adaptive algorithm, this includes the time necessary to compute all solutions $y_{\TT_0},\ldots, y_{\TT_\ell}$, while in the uniform case we just measure the time to compute $y_{\TT}$ for the respective step size.}
    \label{fig:time_ana}
\end{figure}
%\FloatBarrier
\subsection{Nonlinear Heat Equation}
We consider a nonlinear variant of the heat equation with homogeneous Dirichlet boundary conditions, i.e.
\begin{align*}
    \partial_t u - {\rm div}\Big((1+{\rm exp}(-|\nabla u|^2))\nabla u\Big) &= f \quad \text{in} \quad [t_0, \tend] \times \Omega, \\
    u &= 0 \quad \text{on} \quad [t_0, \tend] \times \partial \Omega, \\
    u(0,\cdot)&=u_0 \quad \text{in} \quad \Omega
\end{align*}
for $\Omega=[0,1]^2$, $f=0$. Introducing the weak formulation
\begin{align*}
    \int_\Omega \partial_t u(t) v dx + \int_\Omega \Big((1+{\rm exp}(-|\nabla u(t)|^2)\Big) \nabla u(t) \nabla v dx =0
\end{align*}
for all $v \in H^1_0(\Omega)$, and a semi-discretization in space by means of first order finite-elements $\XX_h \subset H^1_0(\Omega)$, we obtain a ODE-system that fits into the framework of our theory similar to Section~\ref{sec:heatexample}
\begin{align*}
     \int_\Omega \partial_t u(t) v dx + \int_\Omega \Big((1+{\rm exp}(-|\nabla u(t)|^2)\Big) \nabla u(t) \nabla v dx&=0, \quad \text{for all} \quad v \in \XX_h, \\
     u(0)&=u_0 \in \XX_h.
\end{align*}
Again, we consider the initial condition $u_0=1$, and project it onto $\XX_h$ via the $L^2(\Omega)$ projection. The nonlinear problems are solved with a straightforward Newton method.

\begin{figure}[h]
    \centering
    
    \begin{subfigure}{0.45 \textwidth}
        \centering
        \includegraphics[width=\textwidth, trim=25 0 30 30, clip]{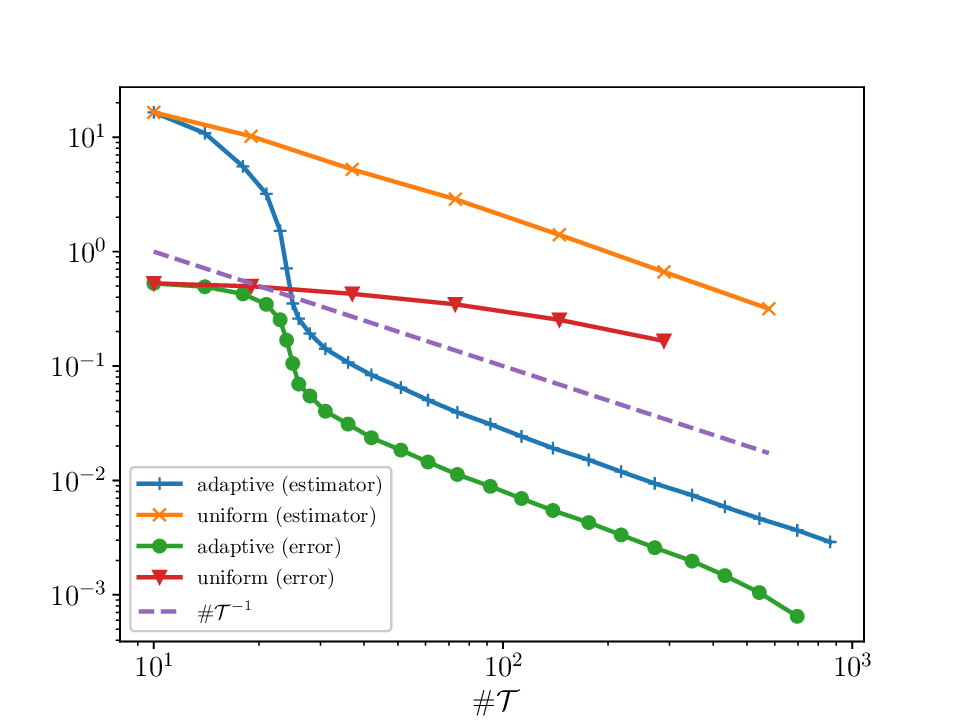}
        \caption{CN discretization, $p=1$}
    \end{subfigure}
    \begin{subfigure}{0.45 \textwidth}
        \centering
         \includegraphics[width=\textwidth, trim=25 0 30 30, clip]{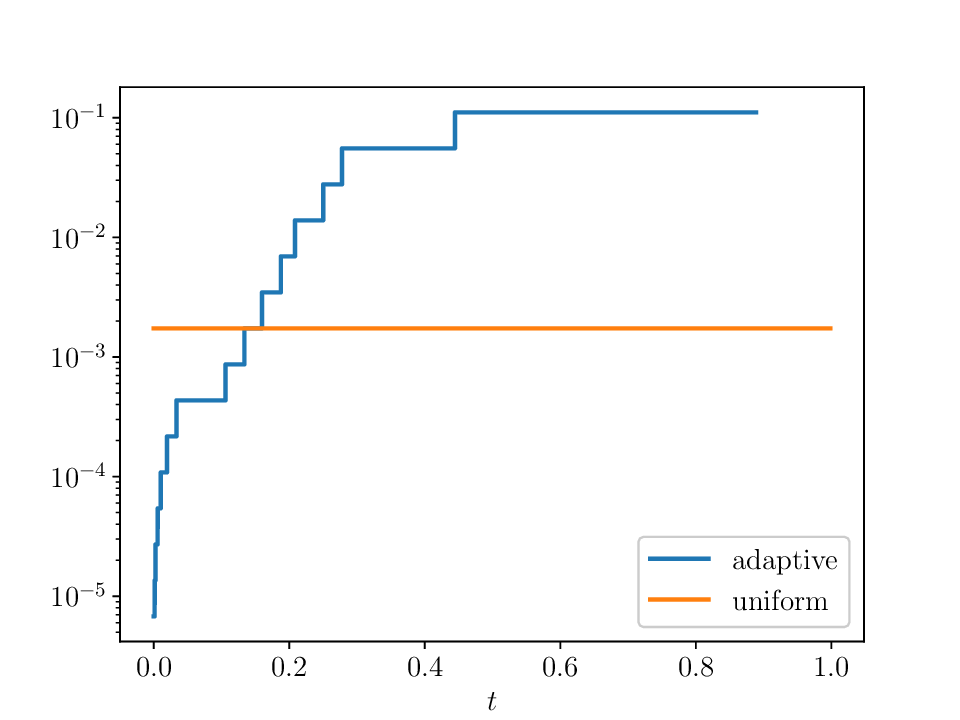}
          \caption{CN discretization, $p=1$}
    \end{subfigure}\\
     \begin{subfigure}{0.45 \textwidth}
        \centering
         \includegraphics[width=\textwidth, trim=25 0 30 30, clip]{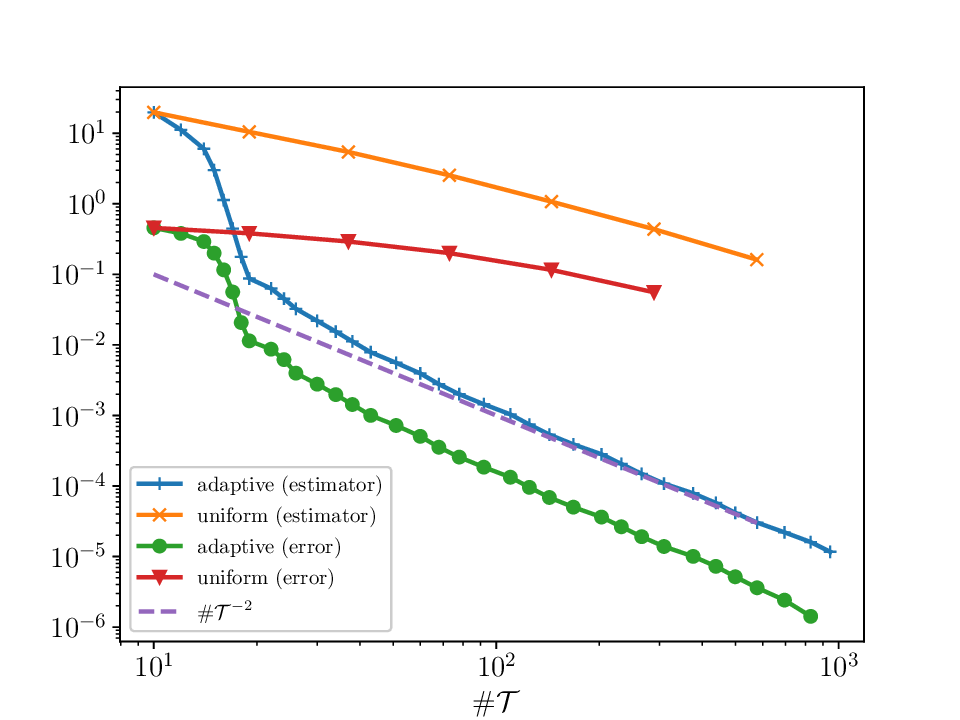}
         \caption{Lobatto discretization, $p=2$}
    \end{subfigure}
    \begin{subfigure}{0.45 \textwidth}
        \centering
         \includegraphics[width=\textwidth, trim=25 0 30 30, clip]{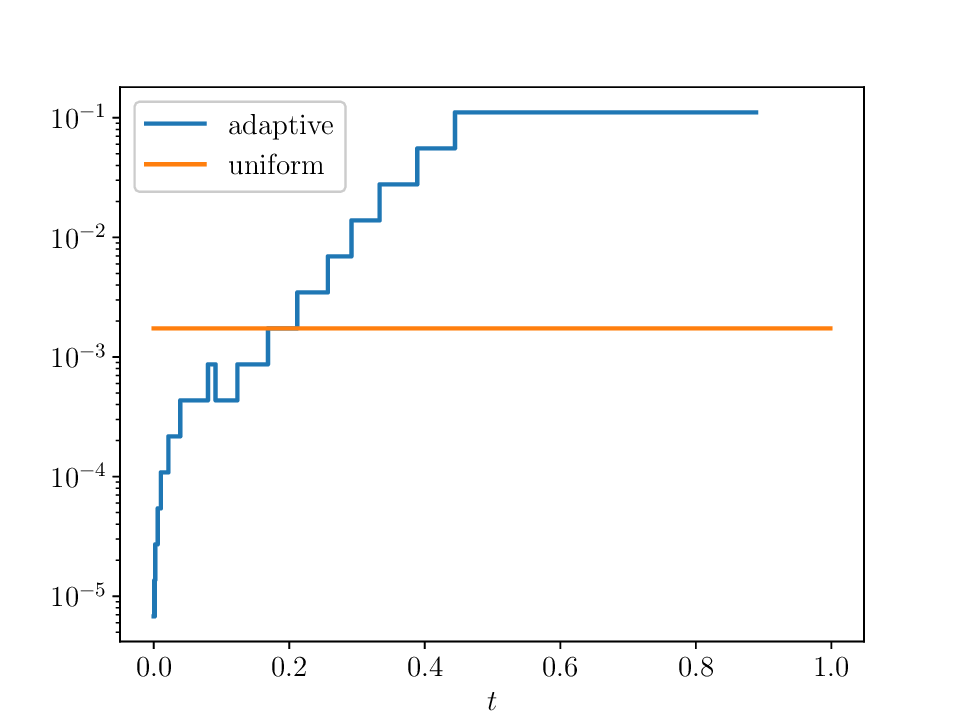}
         \caption{Lobatto discretization, $p=2$}
    \end{subfigure}
    \caption{(Nonlinear heat equation) Left: We plot the $H^1_0$-error $\norm{\partial_t(u-u_\TT)}{L^2((t_0,\tend))}$ and the estimator $\eta_\TT$ over number of time intervals $\#\TT$. The space discretization uses the mesh size $\approx1/10$. Right: We see the adaptive step sizes compared with the uniform step size for the finest computations. The parameter for the adaptive algorithm is $\theta=0.5$.  }
    
    \label{fig:heat_nonlin}
    \end{figure}

    \begin{figure}[h]

    \centering
    \begin{subfigure}{0.45 \textwidth}
        \centering
        \includegraphics[width=\textwidth, trim=5 0 30 30, clip]{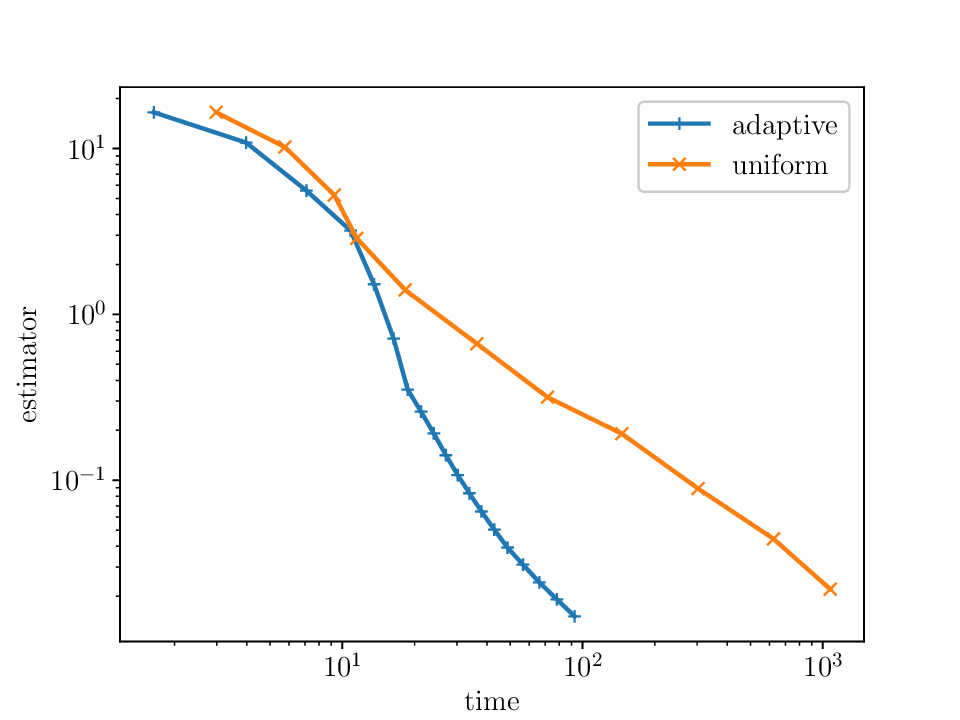}
        \caption{CN discretization, $p=1$}
    \end{subfigure}
    \begin{subfigure}{0.45 \textwidth}
        \centering
         \includegraphics[width=\textwidth, trim=5 0 30 30, clip]{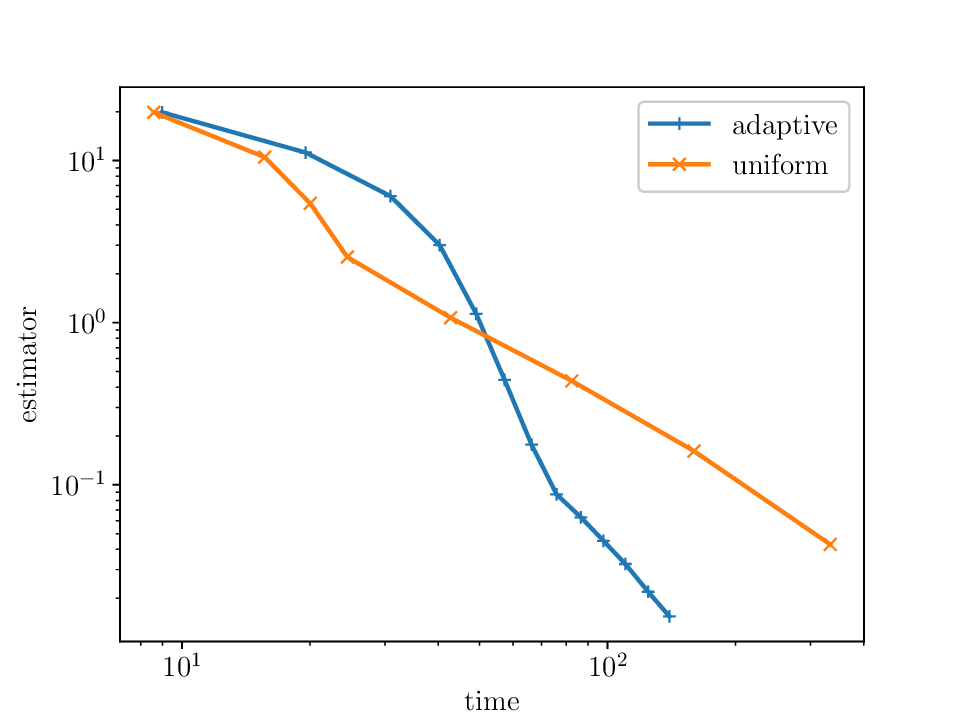}
          \caption{Lobatto discretization, $p=2$}
    \end{subfigure}\\
    \caption{(Nonlinear heat equation) We plot the error estimator $\eta_\TT$ over the total compute time. For the adaptive algorithm, this includes the time necessary to compute all solutions $y_{\TT_0},\ldots, y_{\TT_\ell}$, while in the uniform case we just measure the time to compute $y_{\TT}$ for the respective step size.}
    \label{fig:time_ana_nonlin}
\end{figure}

%\FloatBarrier
\subsection{Van-der-Pol Equation}\label{sec:VDP}
A classical example of a stiff nonlinear ODE system is the Van-der-Pol oscillator given as
\begin{align*}
    \partial_t x&= y,\\
    \partial_t y&= \mu (1-x^2)y-x.
\end{align*}
For this experiment we aim to solve the system on $ [t_0, \tend]=[0,20]$, with $\mu=10$ and $x(0)=y(0)=1$. The error estimator from~\eqref{eq:estimator} reads
\begin{align*}
\eta_\TT(T)^2 = |T|^2\Big(\norm{\partial_t y_\TT-\partial_t^2 x_\TT}{L^2(T)}^2+\norm{(-2\mu x_\TT y_\TT-1)\partial_t x_\TT+ \mu(1-x_\TT^2)\partial_t y_\TT-\partial_t^2y_\TT}{L^2(T)}^2
\Big),
\end{align*}
where the second derivative terms disappear for $p=1$.
We see a clear advantage of the adaptive algorithm over the uniform algorithm, although we do not get a better convergence rate. This is expected, as the exact solution does not show a real singularity, it just becomes steeper with increasing $\mu$. Hence, the adaptive method shows its strength when considering large parameters $\mu$. We particularly emphasize Figure~\ref{fig:roboust}, where we see that the adaptive algorithm is very robust with respect to the parameter $\mu$, while the accuracy of the uniform approach deteriorates significantly. Unfortunately, the theoretical analysis of Theorem~\ref{thm:optim} does not cover this, as the constant $C_{\rm opt}$ depends on the Lipschitz continuity~\eqref{eq:Lip}--\eqref{eq:JacLip} of $F$ and therefore also on $\mu$.
\begin{figure}[h]
    \centering
    
    \begin{subfigure}{0.45 \textwidth}
        \centering
        \includegraphics[width=\textwidth, trim=25 0 30 30, clip]{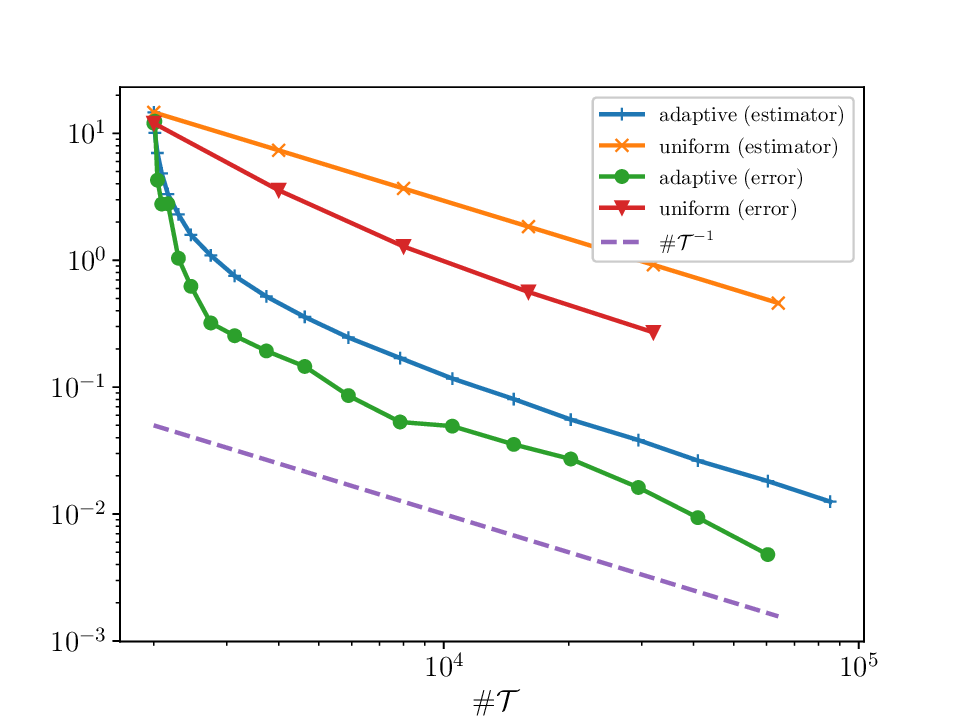}
        \caption{CN discretization, $p=1$}
    \end{subfigure}
    \begin{subfigure}{0.45 \textwidth}
        \centering
         \includegraphics[width=\textwidth, trim=25 0 30 30, clip]{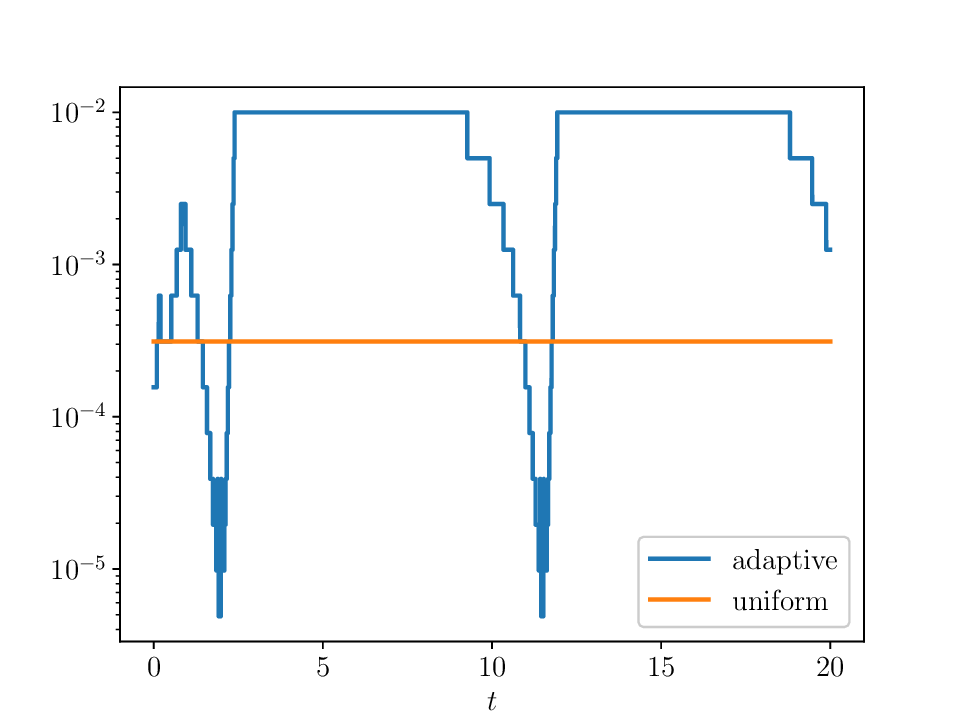}
          \caption{CN discretization, $p=1$}
    \end{subfigure}\\
     \begin{subfigure}{0.45 \textwidth}
        \centering
         \includegraphics[width=\textwidth, trim=25 0 30 30, clip]{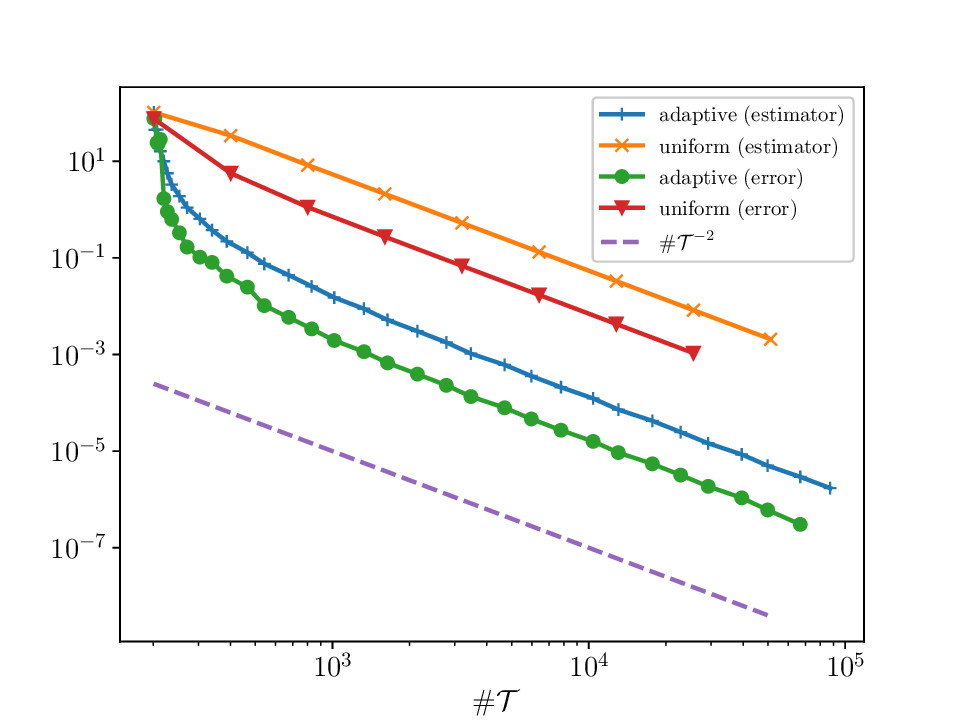}
         \caption{Lobatto discretization, $p=2$}
    \end{subfigure}
    \begin{subfigure}{0.45 \textwidth}
        \centering
         \includegraphics[width=\textwidth, trim=25 0 30 30, clip]{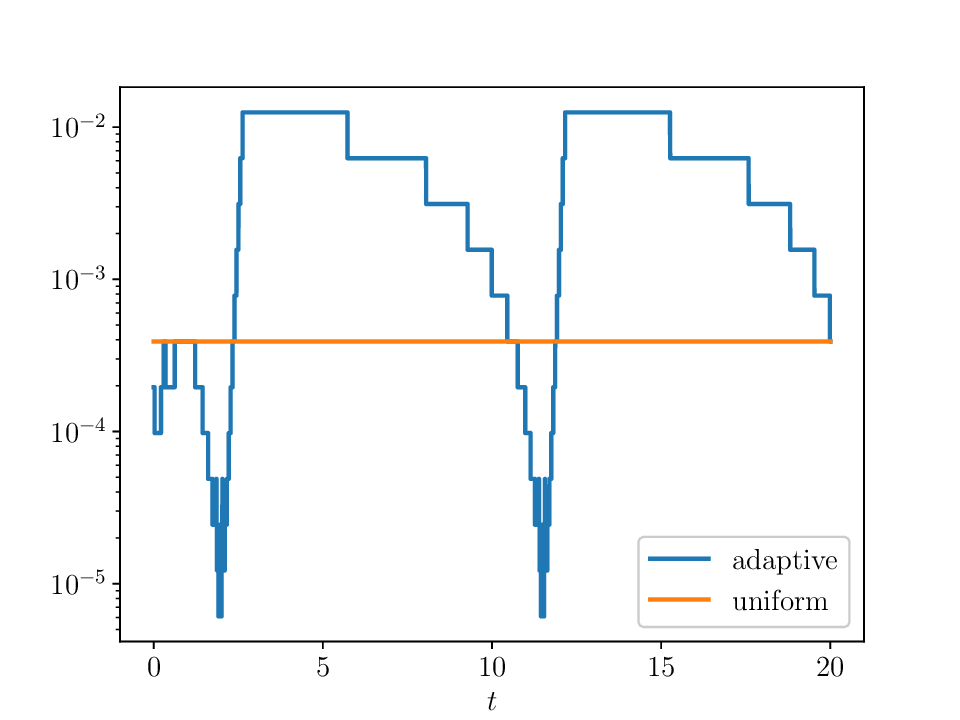}
         \caption{Lobatto discretization, $p=2$}
    \end{subfigure}
    \caption{(Van-der-Pol equation) Left: We plot the $H^1_0$-error $\norm{\partial_t(u-u_\TT)}{L^2((t_0,\tend))}$ and the estimator $\eta_\TT$ over number of time intervals $\#\TT$. Right: We see the adaptive step sizes compared with the uniform step size for the finest computations. The parameter for the adaptive algorithm is $\theta=0.7$.  }
    
    \label{fig:vdp}
    \end{figure}

\begin{figure}
    \begin{center}    
        \includegraphics[width=0.32\textwidth, trim=25 0 30 30, clip]{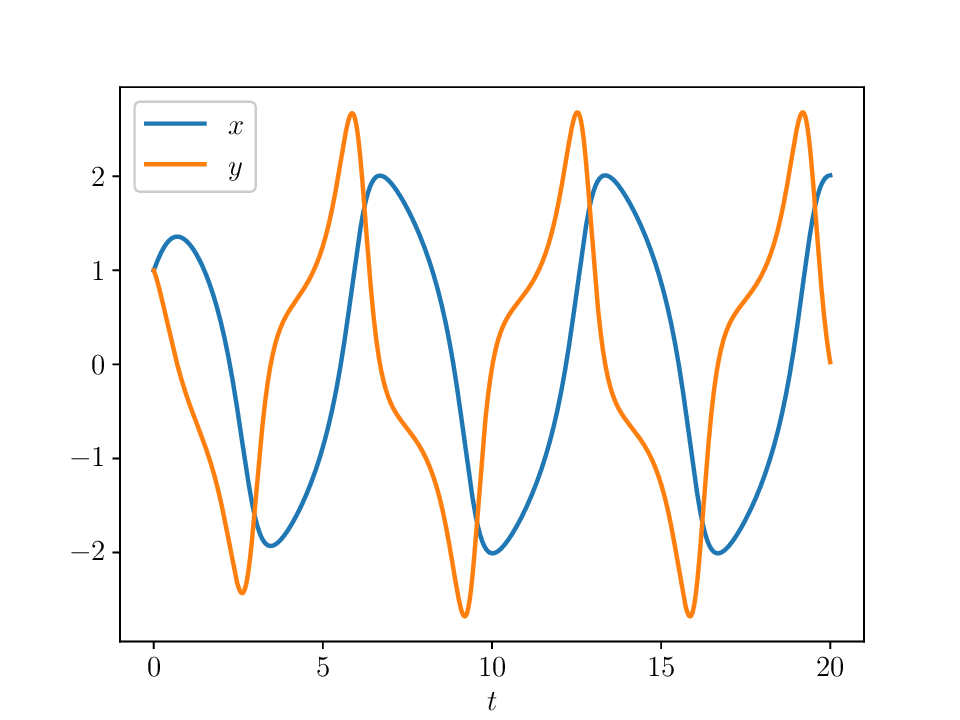}     
         \includegraphics[width=0.32\textwidth, trim=25 0 30 30, clip]{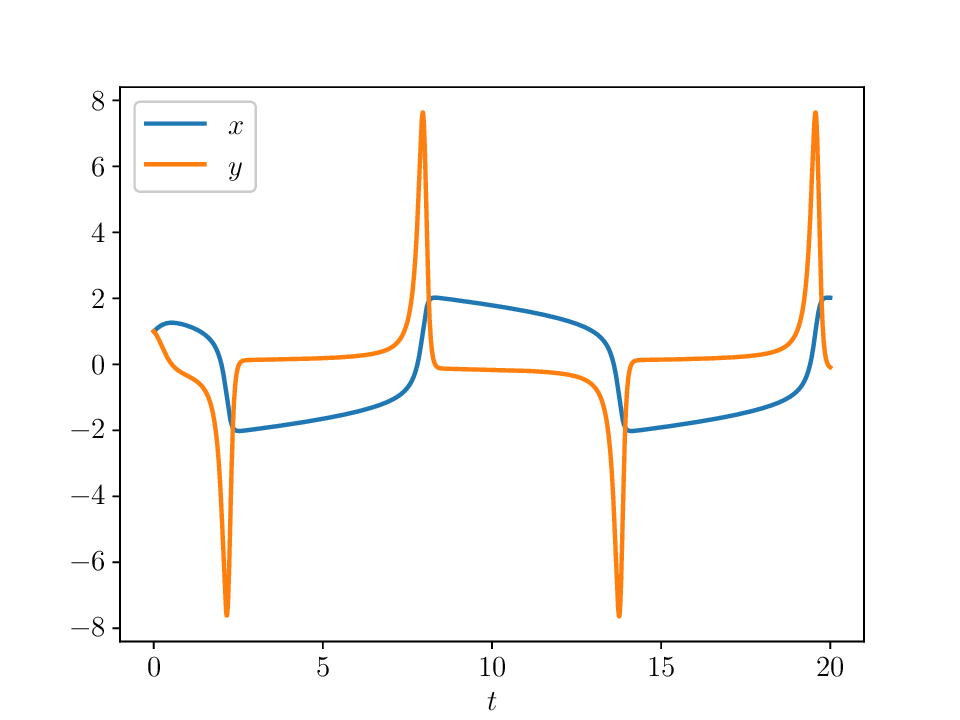}
        \includegraphics[width=0.32\textwidth, trim=25 0 30 30, clip]{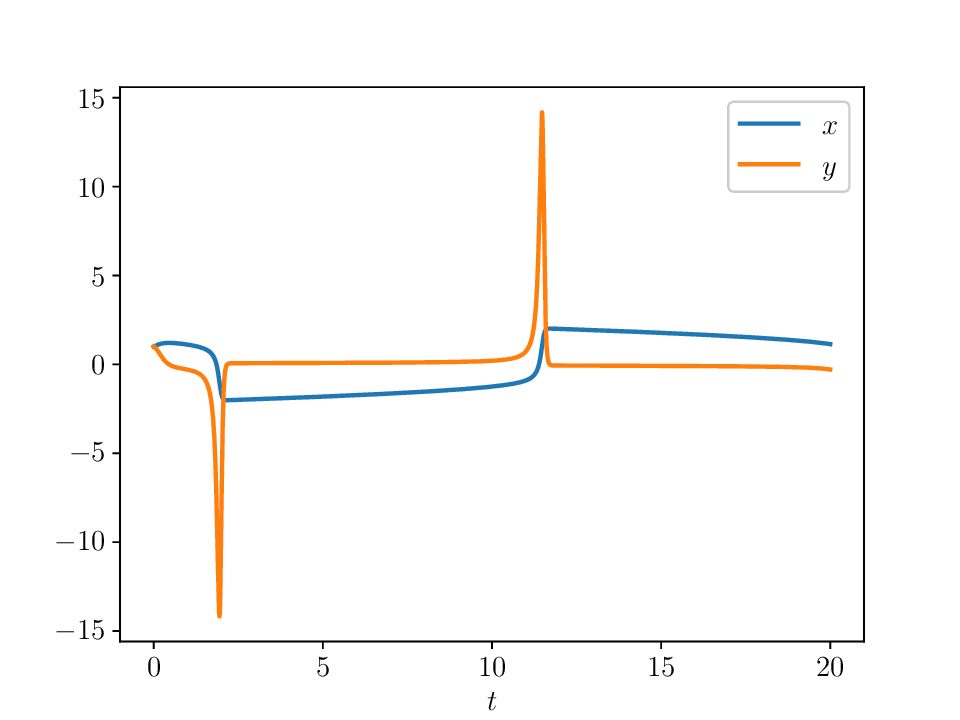}
    \end{center}
    \caption{(Van-der-Pol equation) The solutions of the Van-der-Pol system for $\mu=1,5,10$.}
    \label{fig:sol_vdp}
\end{figure}

\begin{figure}
    \centering
    \begin{subfigure}{0.45 \textwidth}
        \centering
        \includegraphics[width=\textwidth, trim=25 0 30 30, clip]{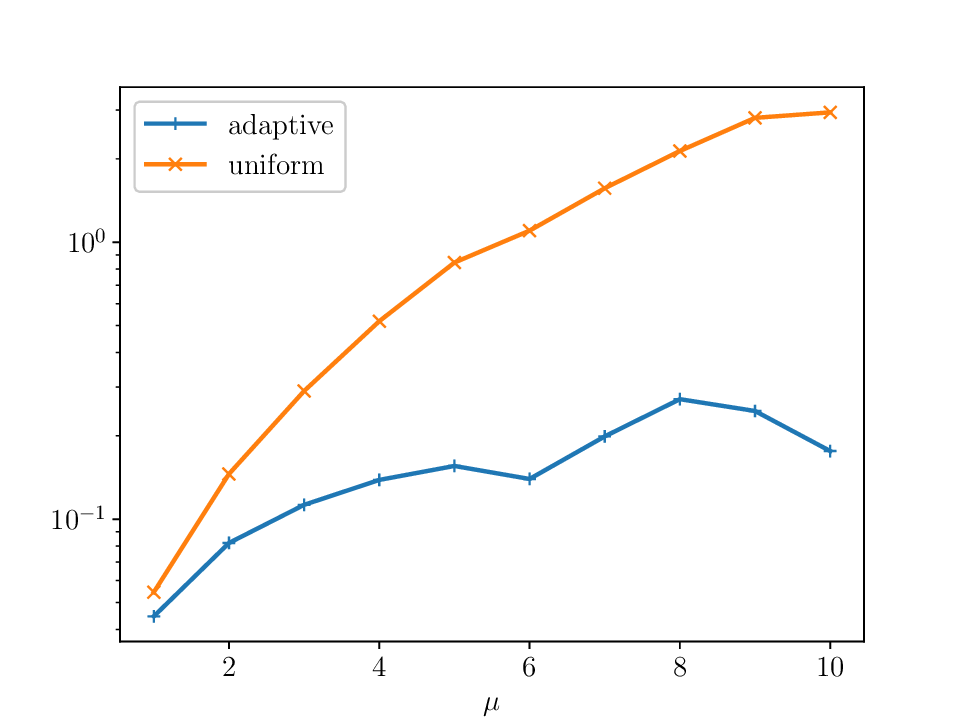}
        \caption{CN}
    \end{subfigure}
    \begin{subfigure}{0.45 \textwidth}
        \centering
        \includegraphics[width=\textwidth, trim=25 0 30 30, clip]{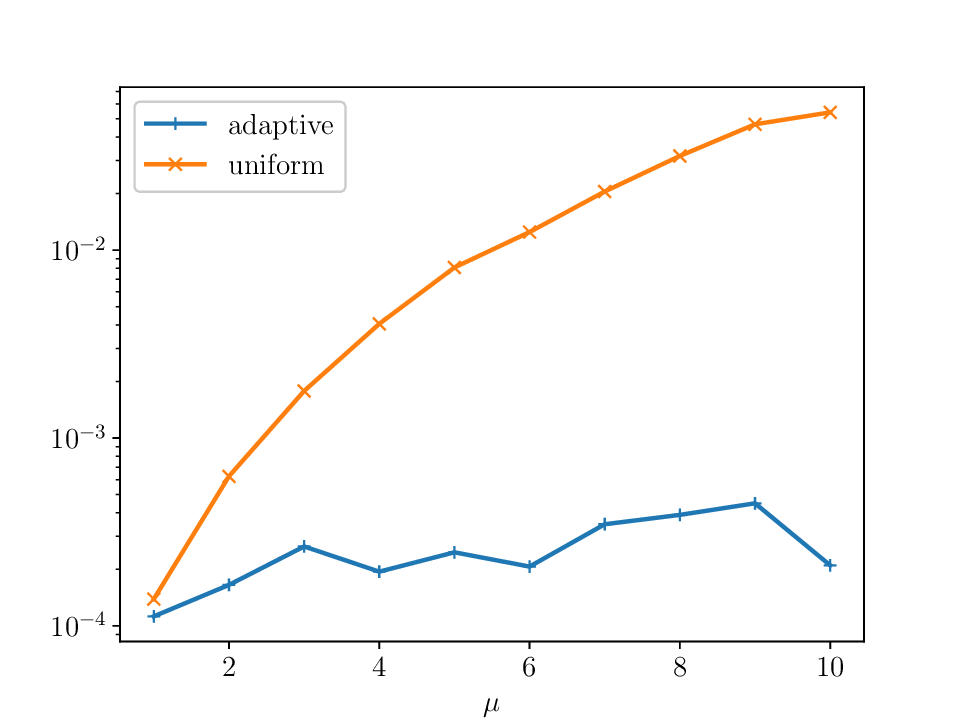}
        \caption{Lobatto}
    \end{subfigure}
    \caption{(Van-der-Pol equation) Uniform and adaptive error estimator with $\approx10000$ time steps for different values of $\mu$. We observe that the adaptive algorithm is very robust with respect to $\mu$, while the uniform approach deteriorates significantly.}
    \label{fig:roboust}
\end{figure}

\subsubsection{Comparison with classical adaptive Radau methods}
To really test the capabilities of the adaptive algorithm, we choose a slight modification of the Van-der-Pol equation, namely we replace the second equation by
\begin{align*}
    \eps\partial_t y&= (1-x^2)y-x
\end{align*}
for a parameter $\eps>0$ that acts similarly to $1/\mu^2$. This modification arises through the rescaling $\widetilde x(t):=x(\mu t)$ and $\widetilde y(t):=\mu y(\mu t)$
and has the advantage that frequency of the typical oscillations of the solution is now independent of $\eps$, which allows us to properly investigate the regime of $\eps\to 0$.
The gold standard for solving this type of problems are the Radau IIA methods developed by Hairer \& Wanner, see e.g.,~\cite{hairer,hairer2} and~\cite{hairer1999}. Their adaptive Radau IIA method is a classical adaptive step-size selection which is based on the comparison of different order approximations within the stages of the Radau schemes, for details we refer to~\cite[Section 6.3]{hairer1999}.

We compare the methods to our adaptive approach (Algorithm~\ref{alg:adaptive} with Radau quadrature, see Section~\ref{sec:quad}) in terms of the error estimator and in terms of the $L^\infty$-error for $\eps=10^{-6}$. Additionally, to the marking in Algorithm~\ref{alg:adaptive}, we refine a time interval whenever the Newton method does not converge within a fixed number of iterations. Although this is not covered by our theory, it helps in finding a feasible initial mesh $\TT_0$.
For the fifth order method, we compare with the \texttt{SciPy} implementation of the adaptive Radau 5 method, which is directly based on the work of Hairer \& Wanner. For the higher order methods, we use our own implementation of the adaptive Radau 9 and Radau 13 methods.

We see in Figures~\ref{fig:radau1}--\ref{fig:radau2} that our adaptive algorithm beats the classical methods in terms of time steps vs. error estimator and $L^\infty$-error.
For the $L^\infty$-comparison, we use the $L^\infty$-error bound from Lemma~\ref{lem:infty_reliability} for marking in Algorithm~\ref{alg:adaptive}, i.e., we use $\eta_{\TT,{\rm max}}(T):=|T|^{1/2}\eta_\TT(T)$ instead of $\eta_\TT(T)$. In terms of compute time, we see in Figure~\ref{fig:radau2}~(B), that a straightforward Python implementation of Algorithm~\ref{alg:adaptive} is only three times slower than the highly optimized \texttt{SciPy} implementation of the adaptive Radau 5 method. We found that marking parameters $\theta\geq 0.9$ give the best runtime performance.

It is remarkable that Algorithm~\ref{alg:adaptive} also achieves the optimal convergence rate for the $L^\infty$-error, which is not covered by our theoretical results.
We did not compare with the adaptive order selection strategy of~\cite{hairer1999}, as we do not have a rigorous way to switch the order of the method optimally in Algorithm~\ref{alg:adaptive}.

We found it useful to stretch the interval by $1/\eps$ and accordingly scale the right-hand side in order to diminish the influence of the $H^1$-seminorm of the error. Interestingly, the stretching seems also to improve the performance of the classical method slightly. It was also essential to use the modified estimator from Section~\ref{sec:modification} for marking in order to avoid over-refinement of later time intervals.

\begin{figure}
    \centering
    \begin{subfigure}{0.45 \textwidth}
        \centering
        \includegraphics[width=\textwidth]{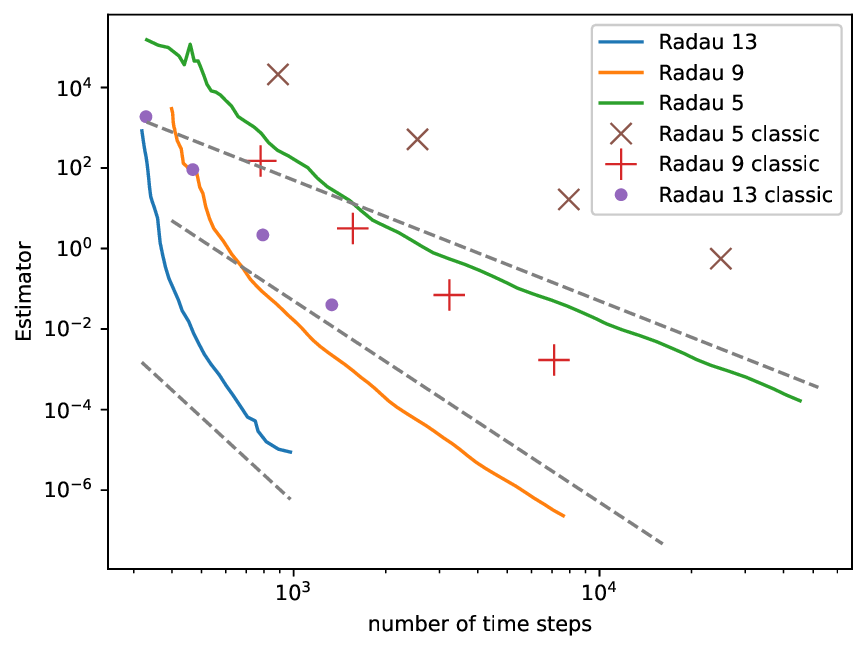} 
        \caption{Algorithm~\ref{alg:adaptive} compared with the classic adaptive Radau methods on $[0,5]$.  The dashed lines represent the respective optimal orders of $3$, $5$ and $7$.}
    \end{subfigure}
    \begin{subfigure}{0.45 \textwidth}
        \centering
        \includegraphics[width=\textwidth]{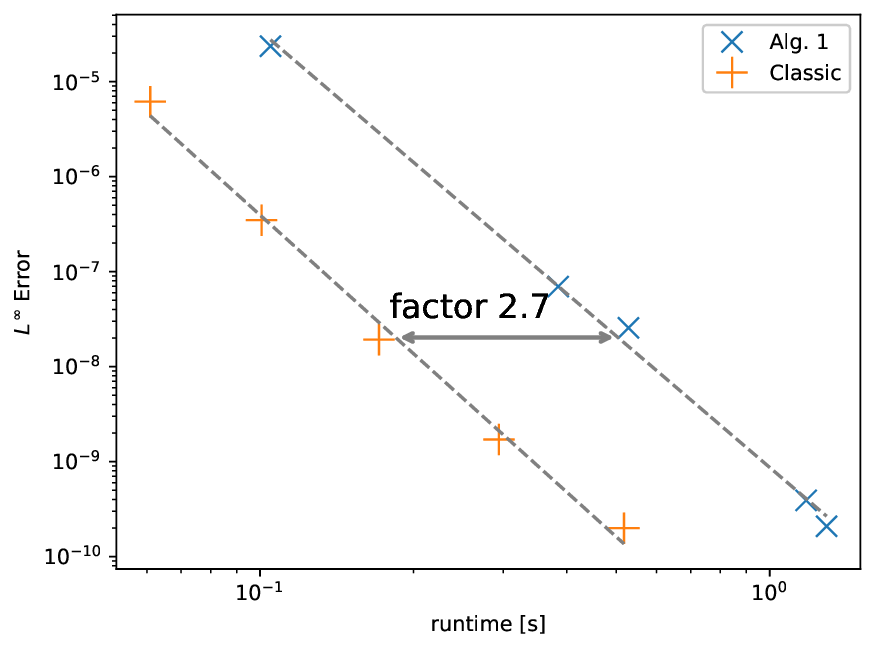}
        \caption{Radau 5 method on $[0,3]$. Runtimes to reach a given $|y-y_\TT|_{\infty}:=\max_{k}|y(k/10)-y_\TT(k/10)|$-error of Algorithm~\ref{alg:adaptive} and the adaptive \texttt{SciPy} implementation.}
    \end{subfigure}
    \caption{(Van-der-Pol equation) We observe that Algorithm~\ref{alg:adaptive} is much superior to the classical adaptive Radau methods in terms of the error estimator $\eta_\TT$.}
    \label{fig:radau1}
\end{figure}
    \begin{figure}
    \centering
    \begin{subfigure}{0.45 \textwidth}
        \centering
        \includegraphics[width=\textwidth]{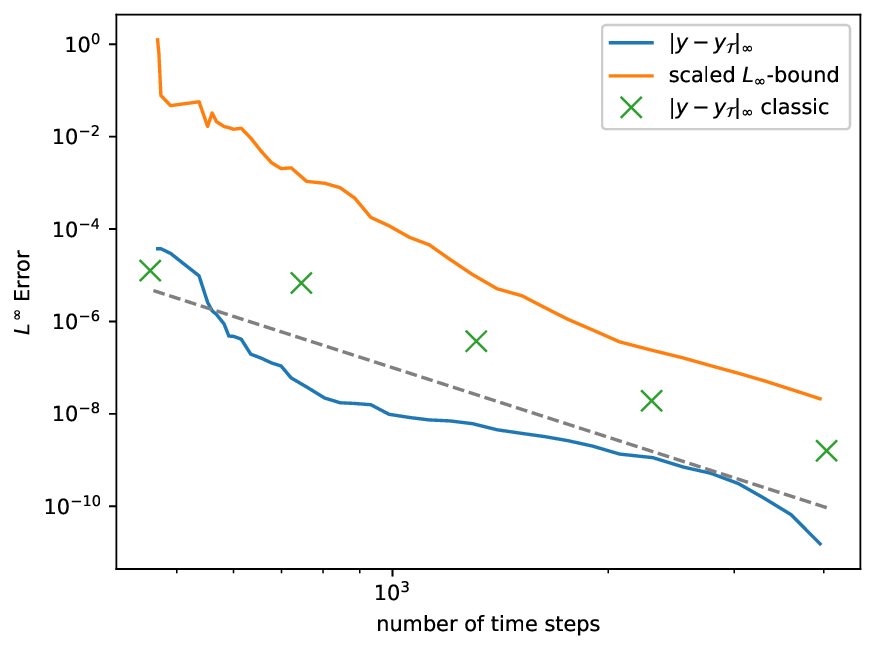} 
        \caption{$[t_0,\tend]=[0,3]$}
    \end{subfigure}       
    \begin{subfigure}{0.45 \textwidth}
        \centering
        \includegraphics[width=\textwidth]{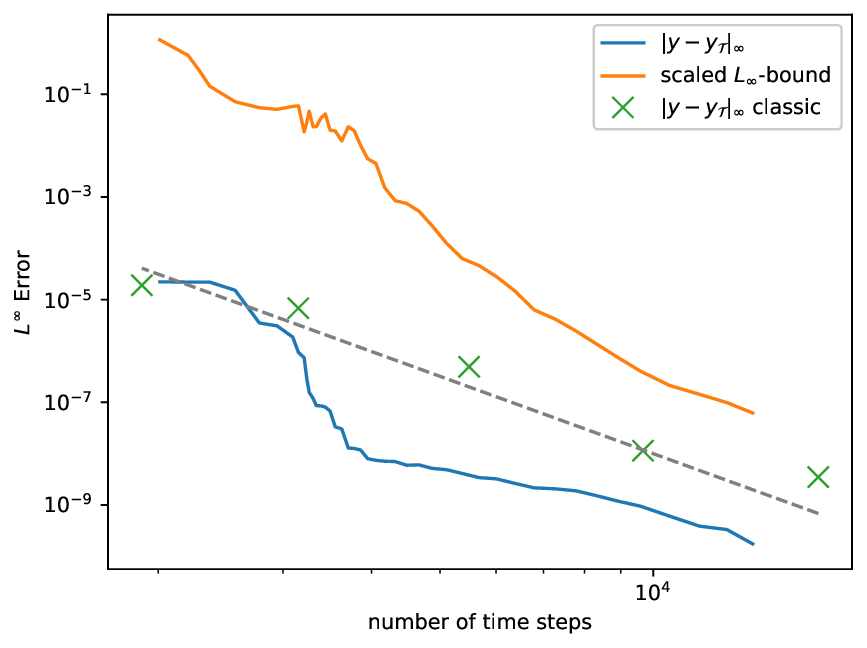} 
        \caption{$[t_0,\tend]=[0,11]$}
    \end{subfigure}
    
    \caption{(Van-der-Pol equation) In terms of $L^\infty$-error, Algorithm~\ref{alg:adaptive} with the Radau 5 method performs slightly better than the classical Radau 5 method on the time intervals $[0,3]$ (left) and $[0,11]$ (right). We measure $|y-y_\TT|_{\infty}:=\max_{k}|y(k/10)-y_\TT(k/10)|$, where $y$ is replaced by a reference solution computed with tolerance $10^{-12}$. The dashed lines represent the respective optimal order of $5$. The orange line represents the error bound from Lemma~\ref{lem:infty_reliability} scaled with $10^{-5}$.}
    \label{fig:radau2}
\end{figure}
%\FloatBarrier
\subsection{Predator-Prey Model}
Another classical example of a stiff nonlinear ODE system is a predator prey model, described by the system 
\begin{align*}
    \partial_t x&= \alpha x- \beta xy,\\
    \partial_t y&= -\gamma y + \delta xy
\end{align*}
For this experiment we aim to solve the system on $[t_0, \tend]=[0,20]$, with $\alpha=1.1$, $\beta=0.4$, $\gamma=0.4$, $\delta =0.1$ and $x(0)=y(0)=10$. The error estimator from $\eqref{eq:estimator}$ reads
\begin{align*}
    \eta_\TT(T)^2=|T|^2\Big(\norm{(\alpha-\beta y_\TT)\partial_t x_\TT-\beta x_\TT \partial_t y_\TT-\partial^2_t x_\TT}{L^2(T)}^2+\norm{\delta y_\TT \partial_t x_\TT+(\delta x_\TT-\gamma)\partial_t y_\TT-\partial^2_t y_\TT}{L^2(T)}^2 \Big),
\end{align*}
where the second derivative terms disappear for $p=1$. Again, we observe a clear advantage of the adaptive approach. As in the previous experiment, there is no improved convergence rate, as the exact solution is smooth.
\begin{figure}[h]
    \centering
    
    \begin{subfigure}{0.45 \textwidth}
        \centering
        \includegraphics[width=\textwidth, trim=25 0 30 30, clip]{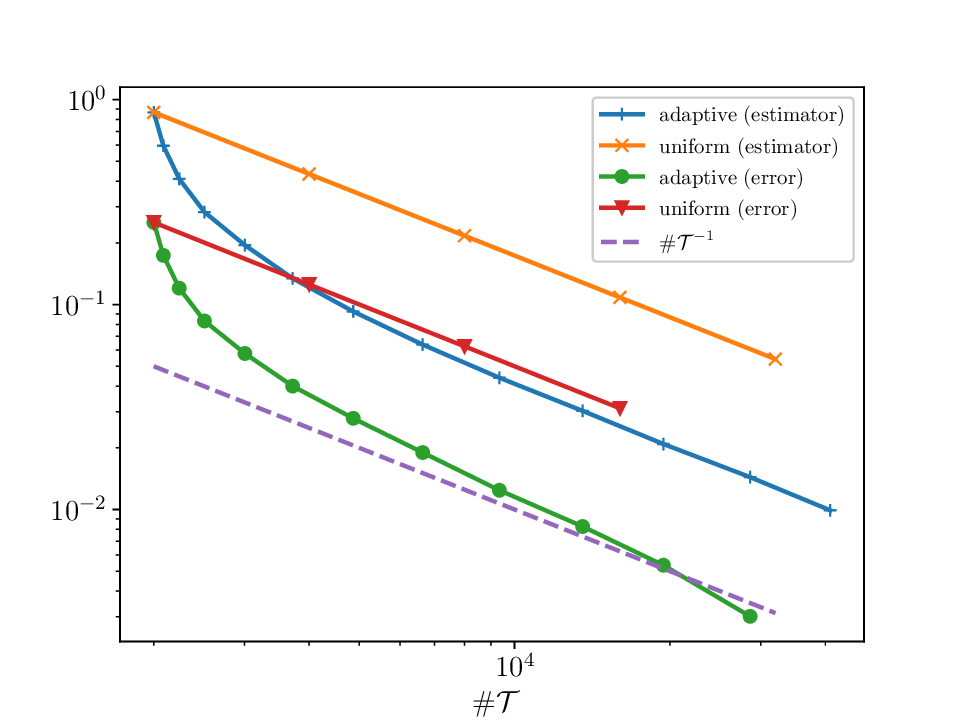}
        \caption{CN discretization, $p=1$}
    \end{subfigure}
    \begin{subfigure}{0.45 \textwidth}
        \centering
         \includegraphics[width=\textwidth, trim=25 0 30 30, clip]{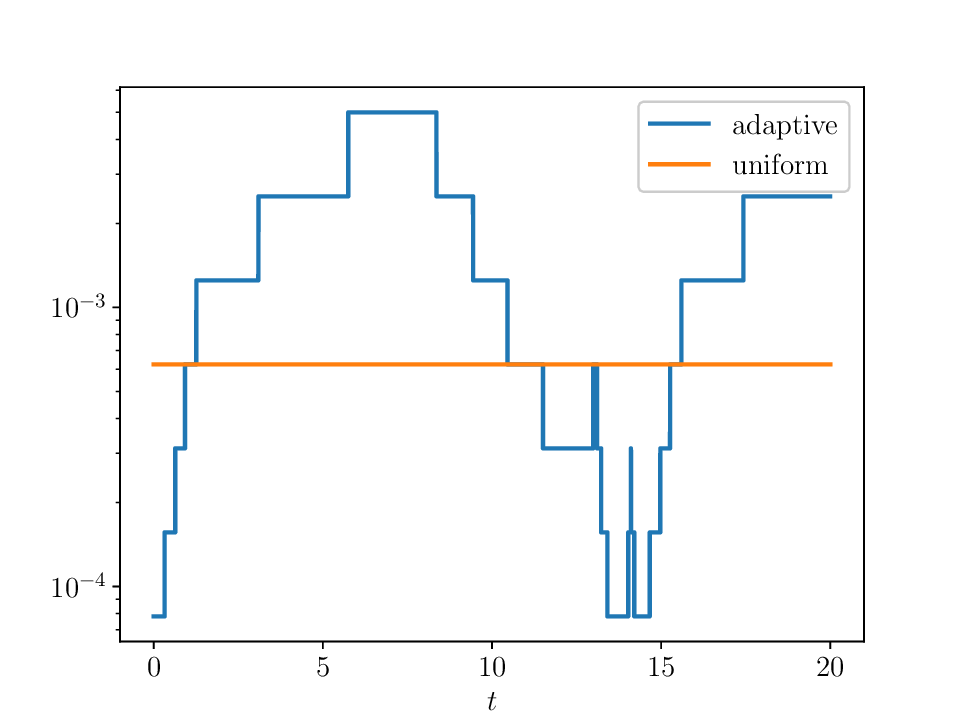}
          \caption{CN discretization, $p=1$}
    \end{subfigure}\\
     \begin{subfigure}{0.45 \textwidth}
        \centering
         \includegraphics[width=\textwidth, trim=25 0 30 30, clip]{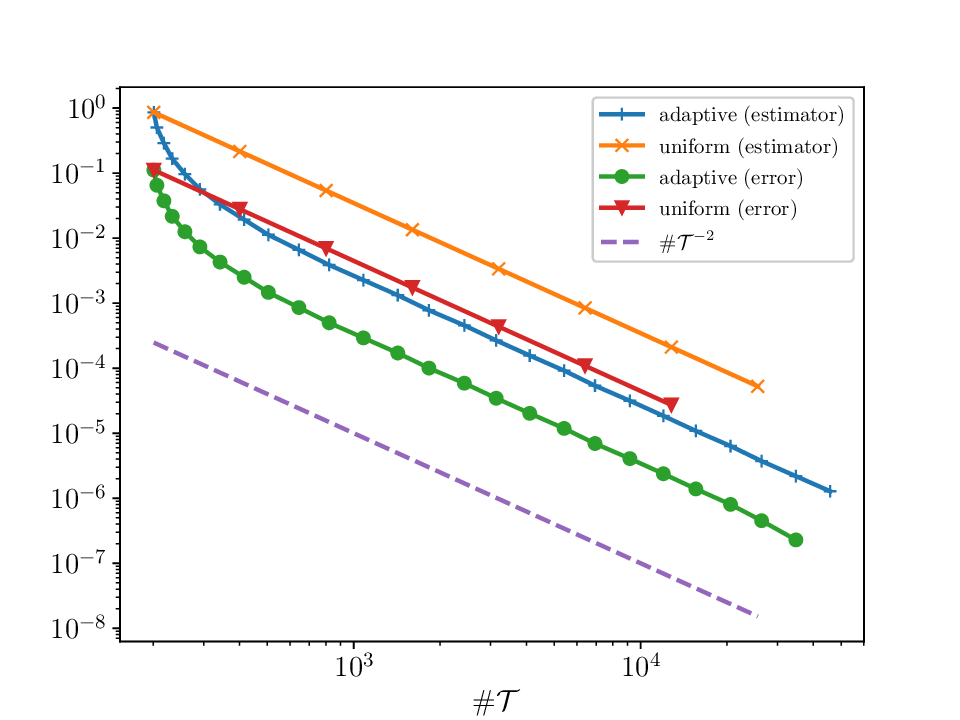}
         \caption{Lobatto discretization, $p=2$}
    \end{subfigure}
    \begin{subfigure}{0.45 \textwidth}
        \centering
         \includegraphics[width=\textwidth, trim=25 0 30 30, clip]{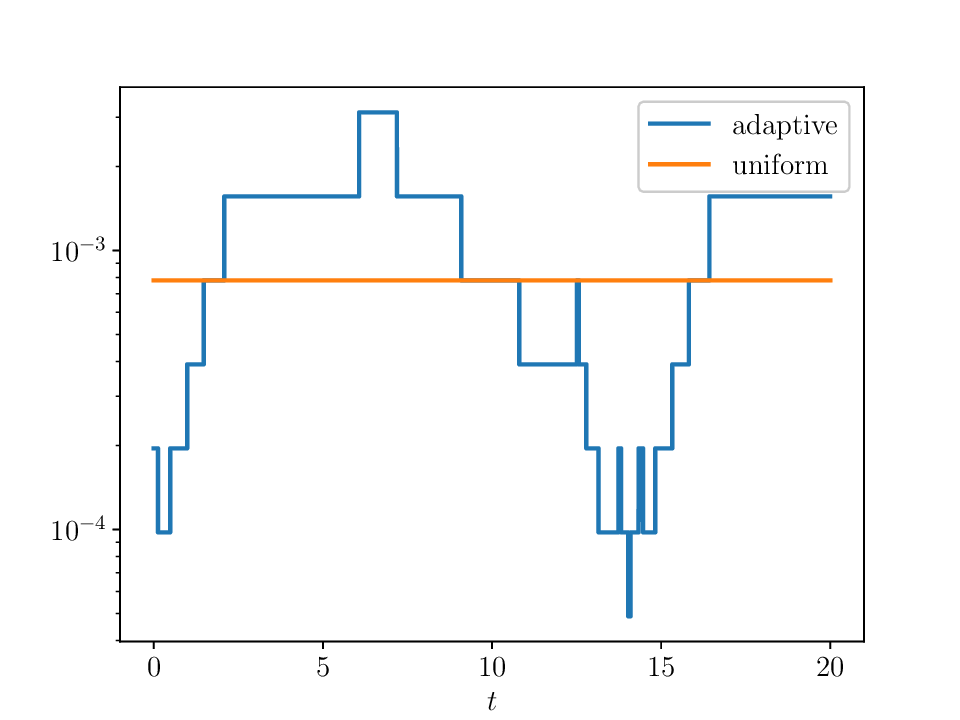}
         \caption{Lobatto discretization, $p=2$}
    \end{subfigure}
    \caption{(Predator-Prey model) Left: We plot the $H^1_0$-error $\norm{\partial_t(u-u_\TT)}{L^2((t_0,\tend))}$ and the estimator $\eta_\TT$ over number of time intervals $\#\TT$. Right: We see the adaptive step sizes compared with the uniform step size for the finest computations. The parameter for the adaptive algorithm is $\theta=0.7$.   }
    
    \label{fig:pp}
    \end{figure}

    \begin{figure}
        \centering
        \includegraphics[width=0.5\linewidth, trim=25 0 30 30, clip]{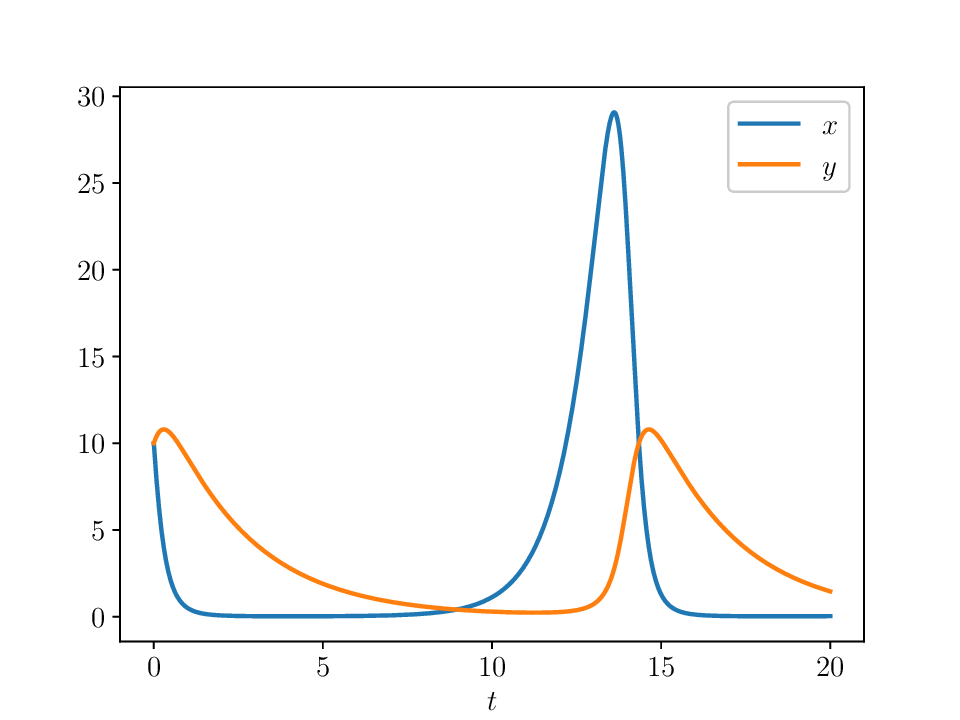}
        \caption{(Predator-Prey model) Solution of the predator prey system.}
        \label{fig:sol_pp}
    \end{figure}
%\FloatBarrier

\section{Proof of Theorem~\ref{thm:optim} and auxiliary results}\label{sec:proof}
In this Section, we will prove Theorem~\ref{thm:optim}. We will first show that the error estimator is an upper bound for the error, i.e., it is reliable. Then, we will show that the adaptive algorithm is optimal in the sense of \eqref{eq:rate_optim} and that the error converges with the same rate as the error estimator.
\subsection{Linear convergence and total time optimality}\label{sec:time}
We will use the framework of~\cite{Carstensen_2014} to show optimality. This method first shows linear convergence of the error estimator, i.e., the output of Algorithm~\ref{alg:adaptive} satisfies
\begin{align}\label{eq:linconv}
    \eta_{\TT_{\ell+k}}\leq C_{\rm lin} q^k \eta_{\TT_\ell}\quad\text{for all }\ell,k\in\N.
\end{align}
This is a crucial intermediate step on the way to optimal convergence rates. Clearly,~\eqref{eq:linconv} ensures that only logarithmically many iterations are needed to reach a desired accuracy. Moreover, it serves as the theoretical justification of the fact that the cumulative compute time for the adaptive algorithm is dominated by the last iteration as shown in the proof of Theorem~\ref{thm:optim}.

\subsection{The error estimator is an upper bound for the error}
The following lemma shows the most fundamental property of the error estimator: It is an upper bound for the error. In a posteriori error estimation, this property is called \emph{reliability}.
\begin{lemma}
\label{lem:gen_reliability}
    If $F$ in \eqref{eq:ode} satisfies \eqref{eq:Lip}, the error estimator~\eqref{eq:estimator} is reliable in the sense
\begin{align*}
    \norm{y-y_{\TT}}{H^1([t_0,\tend])}\leq C_{\rm rel}\eta_{\TT}\quad\text{for all admissible meshes }\TT\in\T,
\end{align*}
where $C_{\rm rel}>0$ is independent of $\TT$.
\end{lemma}
\begin{proof}
We will use variants of Gronwall's lemma. Let us define the residual $R(t):=\partial_t y_{\TT}(t)-F(t,y_{\TT}(t))$. It holds for almost every $t \in (t_0, \tend)$ that
\begin{align*}
    \partial_t y(t)-\partial_t y_{\TT}(t)=F(t,y(t))-F(t,y_{\TT}(t))-R(t).
\end{align*}
Integrating from $t_0$ to $t$ and using condition \eqref{eq:Lip} yields
\begin{align*}
    |y(t)-y_\TT(t)| &\leq \int_{t_0}^t |F(s,y(s))-F(s,y_{\TT}(s))|ds+\int_{t_0}^t |R(s)|ds,\\
    & \leq L_1\int_{t_0}^t |y(s)-y_\TT(s)|ds+\int_{t_0}^t |R(s)|ds.
\end{align*}
Gronwall's lemma~\cite[Prop. 2.1]{emmrich1999discrete} gives
\begin{align}\label{eq:firstest}
    |y(t)-y_\TT(t)| \leq e^{L_1 (t-t_0)}\int_{t_0}^t |R(s)|ds \leq e^{L_1 (\tend-t_0)}\int_{t_0}^{\tend} |R(s)|ds.
\end{align}
Moreover, we get
\begin{align*}
    \norm{y-y_{\TT}}{H^1([t_0,\tend])}^2& \leq C^2_{\rm pc} \norm{ \partial_t y-\partial_t y_{\TT}}{L^2([t_0,\tend])}^2=C^2_{\rm pc}\int_{t_0}^{\tend} \Big(F(t,y(t))-F(t,y_{\TT}(t))-R(t)\Big)^2 dt \\
    & \leq C^2_{\rm pc} 2 L_1^2 \int_{t_0}^{\tend} \Big(y(t)-y_\TT(t) \Big)^2 dt + 2\int_{t_0}^{\tend} R(t)^2 dt,
\end{align*}
where we denote the Poincaré constant used in the first step by $C_{\rm pc}=\sqrt{\tend-t_0}$. With~\eqref{eq:firstest}, we get
\begin{align}\label{eq:resrel}
\begin{split}
 \norm{y-y_{\TT}}{H^1([t_0,\tend])}^2
    & \leq C^2_{\rm pc}2L_1^2 e^{2L_1 (\tend-t_0)}\int_{t_0}^{\tend} \Big(\int_{t_0}^{\tend}|R(s)|ds\Big)^2dt+2\int_{t_0}^{\tend} R(t)^2 dt\\
    & \leq C^2_{\rm pc}\Big(2L_1^2 e^{2L_1 (\tend-t_0)} (\tend-t_0)^2+2\Big) \int_{t_0}^{\tend} R(t)^2 dt.
    \end{split}
\end{align}
In~\eqref{eq:gen_disc}, we may choose $v \in P^0(\TT)$ to show $\int_TR(t)dt=0$ for all $T\in\TT$. Thus, we can use a Poincaré type inequality as in~\cite{ANASTASSIOU20081102} to see
\begin{align*}
    \norm{y-y_{\TT}}{H^1([t_0,\tend])}^2 \leq C_{\rm rel}^2 \sum_{T \in \TT} |T|^2 \norm{\partial_t F(\cdot,y_{{\TT}})+\nabla_yF(\cdot, y_{{\TT}}) \partial_t y_{{\TT}}-\partial^2_t y_\TT}{L^2(T)}^2.
\end{align*}
This concludes the proof with $C^2_{\rm rel}=C_{\rm pc}^2\Big(2L_1^2 e^{2L_1 (\tend-t_0)} (\tend-t_0)^2+2\Big).$
\end{proof}
With a small refinement of the proof, we can also obtain a computable upper bound for the pointwise error.
\begin{lemma}\label{lem:infty_reliability}
    If $F$ in \eqref{eq:ode} satisfies \eqref{eq:Lip}, the error estimator~\eqref{eq:estimator} satisfies
\begin{align*}
 \norm{y-y_{\TT}}{L^\infty([t_0,\tend])}\leq e^{L_1(\tend-t_0)}\max_{T\in\TT}|T|^{1/2}\eta_\TT(T).
\end{align*}
\end{lemma}
\begin{proof}
As in the proof of Lemma~\ref{lem:gen_reliability}, we obtain
\begin{align*}
    |y(t)-y_\TT(t)| 
    & \leq L_1\int_{t_0}^t |y(s)-y_\TT(s)|ds+\Big|\int_{t_0}^t R(s)\,ds\Big|.
\end{align*}
Since $\int_TR(t)dt=0$ for all $T\in\TT$, we have $\int_{t_0}^t R(s)\,ds = \int_{t_-}^t R(s)\,ds$, where $t_-$ is the left endpoint of the time interval containing $t$. A Poincar\'e inequality shows
\begin{align*}
\Big|\int_{t_0}^t R(s)\,ds\Big|\leq \max_{T\in\TT} |T|^{3/2}\norm{\partial_t R}{L^2(T)}
\end{align*}
and Gronwall's lemma implies
\begin{align}\label{eq:firstest}
    |y(t)-y_\TT(t)| \leq e^{L_1 (t-t_0)}\max_{T\in\TT} |T|^{3/2}\norm{\partial_t R}{L^2(T)}.
\end{align}
\end{proof}

\subsection{Convergence of uniform time stepping}
Since the proof technique is very similar to the previous result, we make a short excursion to show that if the step size is sufficiently small, also the approximation error can be arbitrarily small. 
\begin{lemma}\label{lem:unifconv}
All approximate solutions $y_\TT$ are uniformly bounded in the sense $\norm{y_\TT}{H^1([t_0,\tend])}\leq C$ for a constant $C<\infty$ that depends only on the exact solution $y$, $\tend$ and $L_1$. Moreover, let $\tau_\TT:= \max_{T\in\TT}|T|$ denote the maximal step size of $\TT$. There holds 
    \begin{align*}
        \lim_{\tau_\TT\to 0 }\norm{y-y_\TT}{H^1([t_0,\tend])} = 0.
    \end{align*}
\end{lemma}
\begin{proof}
We introduce the $L^2$-orthogonal projection onto piecewise polynomials of degree $p-1$, denoted by $\Pi_{p-1}: L^2([t_0,\tend],\R^d) \to \PP^{p-1}(\TT)$. From the definition, we obtain the identity
\begin{align*}
   \partial_t y - \partial_t y_\TT(t) &= F(t,y(t)) - \partial_t y_\TT(t) =  (1-\Pi_{p-1})F(t,y(t)) +\Pi_{p-1}F(t,y(t))- \partial_t y_\TT(t) \\
   &= 
    (1-\Pi_{p-1})F(t,y(t)) +\Pi_{p-1}(F(\cdot,y)- F(\cdot,y_\TT)).
\end{align*}
Lipschitz's continuity of $F$ and integration from $t_0$ to $t$ show
\begin{align*}
    |y(t) - y_\TT(t)|\leq \int_{t_0}^t(1-\Pi_{p-1})F(\cdot,y)\,ds  + L_1 \int_{t_0}^t |y(s)-y_\TT(s)|\,ds.
\end{align*}
Hence, Gronwall's lemma~\cite[Prop. 2.1]{emmrich1999discrete} shows
\begin{align*}
    |y(t) - y_\TT(t)|\leq e^{L_1(t-t_0)}\int_{t_0}^t|(1-\Pi_{p-1})F(\cdot,y)|\,ds
    \leq  e^{L_1(t-t_0)}\sqrt{|t-t_0|}\norm{(1-\Pi_{p-1})F(\cdot,y)}{L^2([t_0,\tend])}.
\end{align*}
Just as in the proof of Lemma~\ref{lem:gen_reliability}, we can use this estimate to also get
\begin{align*}
    \norm{y-y_\TT}{H^1([t_0,\tend])} &\leq C_{\rm pc} \norm{\partial_t y - \partial_t y_\TT}{L^2([t_0,\tend])}\\
    &\leq 
    C_{\rm pc} \norm{(1-\Pi_{p-1})F(t,y(t))}{L^2([t_0,\tend])}+C_{\rm pc}\norm{\Pi_{p-1}(F(\cdot,y)- F(\cdot,y_\TT))}{L^2([t_0,\tend])}\\
    &\leq 
    C_{\rm pc} \norm{(1-\Pi_{p-1})F(t,y(t))}{L^2([t_0,\tend])}+C_{\rm pc}L_1\norm{y- y_\TT}{L^2([t_0,\tend])}\\ 
    &\lesssim \norm{(1-\Pi_{p-1})F(\cdot,y)}{L^2([t_0,\tend])},
\end{align*}
where the hidden constant in the last estimate depends on $C_{\rm pc}$, $\tend$, and $L_1$.
From this, we immediately get boundedness of $\norm{y_\TT}{H^1([t_0,\tend])}$. Moreover, since $y\in H^1([t_0,\tend],\R^d)$ and $F$ is Lipschitz continuous, there  holds $\norm{(1-\Pi_{p-1})F(\cdot,y)}{L^2([t_0,\tend])}\to 0$ as $\tau_\TT \to 0$.
This concludes the proof.
\end{proof}
\subsection{Optimality properties of the error estimator}
We aim to use the framework established in~\cite{Carstensen_2014} and therefore require certain properties of the error estimator called (A1)--(A4) in~\cite{Carstensen_2014}.
The next result is called \emph{Stability on non-refined elements}~(A1) and shows that the error estimator is Lipschitz continuous. The term \emph{element} comes from the fact that the framework in~\cite{Carstensen_2014} was built for stationary mesh refinement where one typically refers to the elements of a mesh. In our context, it would be more accurate to speak of \emph{interval}, however, for consistency we adopt the terminology.
\begin{lemma}
\label{lem:gen_stability}
    If the right-hand side $F$ from~\eqref{eq:ode} satisfies \eqref{eq:Lip}--\eqref{eq:JacLip}, the error estimator~\eqref{eq:estimator} satisfies stability on non-refined elements in the sense that
    \begin{align*}
        \Bigg| \Big(\sum_{T \in \TT \cap \widehat\TT}\eta_{\TT}(T)^2 \Big)^{1/2}-\Big(\sum_{T \in \TT \cap \widehat\TT} \eta_{\widehat\TT}(T)^2\Big)^{1/2} \Bigg| \leq C_{\rm stab} \norm{y_{\TT}-y_{\widehat\TT}}{H^1([t_0,\tend])},
    \end{align*}
     for a refinement $\widehat\TT$ of $\TT \in \T.$
\end{lemma}
\begin{proof}
    Repeated use of the triangle inequality yields
    \begin{align*}
        &\Bigg(\Big(\sum_{T \in \TT \cap \widehat\TT}\eta_{\TT}(T)^2 \Big)^{1/2}-\Big(\sum_{T \in \TT \cap \widehat\TT} \eta_{\widehat\TT}(T)^2\Big)^{1/2}\Bigg)^2 \\
        &\leq \sum_{T \in \TT \cap \widehat\TT}|T|^2\norm{\partial_t F(\cdot,y_{{\TT}})+\nabla_yF(\cdot, y_{{\TT}}) \partial_t y_{{\TT}}-\partial^2_t y_\TT-\partial_t F(\cdot,y_{\widehat\TT})-\nabla_yF(\cdot, y_{\widehat\TT}) \partial_t y_{\widehat\TT}+\partial^2_t y_{\widehat\TT}}{L^2(T)}^2 \\
        &\lesssim \sum_{T \in \TT \cap \widehat\TT}|T|^2\Bigg( \int_T |\partial_t F(t,y_{{\TT}}(t))-\partial_t F(t,y_{\widehat\TT}(t))|^2dt+\int_T|\nabla_yF(t, y_{{\TT}}(t)) \partial_t y_{{\TT}}(t)-\nabla_yF(t, y_{\widehat\TT}(t)) \partial_t y_{\widehat\TT}(t)|^2dt\\
        &\qquad +\int_T |\partial^2_t y_\TT(t)-\partial^2_t y_{\widehat\TT}(t)|^2dt\Bigg).
    \end{align*}
    Because of the Lipschitz condition on the Jacobian of $F$, we get for the first term of the sum
    \begin{align*}
        \int_T |\partial_t F(t,y_{{\TT}}(t))-\partial_t F(t,y_{\widehat\TT}(t))|^2dt \leq L_2^2\int_T |y_{\TT}(t)-y_{\widehat\TT}(t)|^2dt.
    \end{align*}
    For the second term we get
    \begin{align*}
        \int_T|\nabla_y&F(t, y_{{\TT}}(t)) \partial_t y_{{\TT}}(t)-\nabla_yF(t, y_{\widehat\TT}(t)) \partial_t y_{\widehat\TT}(t)|^2dt \\
        &\lesssim \int_T|\nabla_yF(t, y_{{\TT}}(t))-\nabla_yF(t, y_{\widehat\TT}(t))|^2|\partial_t y_{{\TT}}(t)|^2+|\nabla_yF(t, y_{\widehat\TT}(t))|^2|\partial_t y_{{\TT}}(t)-\partial_t y_{\widehat\TT}(t)|^2 dt.
    \end{align*}
    Note that $|T|^2|\partial_t y_{{\TT}}(t)|^2$ and $|\nabla_yF(t, y_{\widehat\TT}(t))|^2$ are uniformly bounded because of Lemma \ref{lem:unifconv}, with a bound only depending on the true solution, the initial discretization $\TT_0$ and $L_2$. Therefore, we obtain
    \begin{align*}
         \int_T|\nabla_yF(t, y_{{\TT}}(t)) \partial_t y_{{\TT}}(t)-\nabla_yF(t, y_{\widehat\TT}(t)) \partial_t y_{\widehat\TT}(t)|^2dt \lesssim  \int_T | y_{\TT}(t)-y_{\widehat\TT}(t)|^2 dt +\int_T |\partial_t y_{\TT}(t)-\partial_t y_{\widehat\TT}(t)|^2dt
    \end{align*}
    For the last term of the sum, an inverse estimate on $T\in \TT\cap \widehat \TT$ gives
    \begin{align*}
        \int_T |\partial^2_t y_\TT(t)-\partial^2_t y_{\widehat\TT}(t)|^2dt \lesssim \frac{1}{|T|^2}\int_T|\partial_t y_\TT(t)-\partial_t y_{\widehat\TT}(t)|^2dt
    \end{align*}
    Combining the estimates and using a Poincaré inequality (since $y_\TT(t_0)=y_{\widehat\TT}(t_0)$) yields the result.
\end{proof}
The next result is called~\emph{Reduction on refined elements}~(A2) in \cite{Carstensen_2014} and shows that, up to a perturbation, the error estimator gets smaller when the step size is reduced.
\begin{lemma}
\label{lem:gen_reduction}
     If the right-hand side $F$ from~\eqref{eq:ode} satisfies \eqref{eq:Lip}--\eqref{eq:JacLip}, the error estimator~\eqref{eq:estimator} satisfies reduction on refined elements in the sense
    \begin{align*}
        \sum_{T \in \widehat\TT \setminus \TT} \eta_{\widehat\TT}(T)^2 \leq q \sum_{T \in  \TT\setminus \widehat\TT} \eta_{\TT}(T)^2 + C_{\rm red}\norm{y_{\TT}-y_{\widehat\TT}}{H^1([t_0,\tend])}^2
    \end{align*}
     for a refinement $\widehat\TT$ of $\TT \in \T,$ $0<q<1$ and $C_{\rm red}>1$.
\end{lemma}
\begin{proof}
    The proof is similar to the one in \cite{generalqo}. Let $T \in \TT\setminus \widehat\TT$ and $T_1,\ldots,T_n \in \widehat\TT \setminus \TT$ such that $T_1 \cup\ldots \cup T_n=T$. It holds that $|T_i|\leq |T|/2$ for all $i=1,\ldots,n$. Now we have
    \begin{align*}
        &\sum_{i=1}^n\eta_{\widehat\TT}(T_i)^2\\
        &\leq (1+ \delta) \sum_{i=1}^n\frac{|T|^2}{4}\norm{\partial_t F(\cdot,y_{\TT})+\nabla_yF(\cdot, y_{\TT}) \partial_t y_{\TT}-\partial^2_t y_\TT}{L^2(T_i)}^2\\
        &\qquad +(1+\delta^{-1}) \sum_{i=1}^n|T_i|^2\norm{\partial_t F(\cdot,y_{{\TT}})+\nabla_yF(\cdot, y_{{\TT}}) \partial_t y_{{\TT}}-\partial^2_t y_\TT-\partial_t F(\cdot,y_{\widehat\TT})-\nabla_yF(\cdot, y_{\widehat\TT}) \partial_t y_{\widehat\TT}+\partial^2_t y_{\widehat\TT}}{L^2(T)}^2
    \end{align*}
    for any $\delta>0$. By using the same estimates as in Lemma \ref{lem:gen_stability} and choosing $\delta>0$ sufficiently small, we get
    \begin{align*}
        \sum_{i=1}^n\eta_{\widehat\TT}(T_i)^2 \leq q_{\delta}\eta_{\TT}(T)^2+ C_{\delta} \norm{y_{\TT}-y_{\widehat\TT}}{H^1(T)}^2.
    \end{align*}
    Summing up over $T\in \TT\setminus \widehat\TT $ yields the result.
\end{proof}
The next result is a refinement of \emph{reliability} (Lemma~\ref{lem:gen_reliability}) and is called \emph{Discrete reliability}~(A4) in~\cite{Carstensen_2014}.
It shows that the difference between to discrete approximation can be bounded on a subset of the domain.
The main idea of the proof of the next statement is an inductive argument, where refined and non-refined elements are treated differently.
\begin{lemma}
\label{lem:gen_reldisc}
  If the right-hand side $F$ from~\eqref{eq:ode} satisfies \eqref{eq:Lip}--\eqref{eq:JacLip} and the initial discretization $\TT_0$ is fine enough, the error estimator~\eqref{eq:estimator} satisfies discrete reliability, in the sense
    \begin{align*}
        \norm{y_{\widehat\TT}-y_{\TT}}{H^1([t_0,\tend])}^2 \leq C_{\rm rel} \sum_{T \in \TT \setminus \widehat\TT}\eta_{\TT}(T)^2
    \end{align*}
    for a refinement $\widehat\TT$ of $\TT \in \T$.
\end{lemma}
\begin{proof}
    Recall the residual $R(t):=\partial_t y_{\TT}(t)-F(t,y_{\TT})$. We have for almost all $t \in (t_0, \tend)$
    \begin{align}
        \label{eq:starting_point}
        \partial_t y_\TT(t)-\partial_t y_{\widehat\TT}(t)=F(t,y_\TT(t))-\partial_t y_{\widehat\TT}(t)+R(t).
    \end{align}
    \emph{Step~1: }
    Let $[t_i,t_{i+1}]=T \in \widehat\TT \setminus \TT$ be a refined element. Since $y_{\TT}-y_{\widehat\TT}$ is a polynomial of order $p$ in $T$, we can uniquely determine it by the values of the first derivative at the $p$ Gauss-Legendre integration points in $T$ and its value at $t_i$, i.e, we write it as
    \begin{align}
    \label{eq:parametrize}
        y_{\TT}(t)-y_{\widehat\TT}(t)=y_{\TT}(t_i)-y_{\widehat\TT}(t_i)+ \int_{t_i}^t \sum_{k=1}^p \partial_t (y_\TT(\sigma_k)-y_{\widehat\TT}(\sigma_k)) l_k(s) ds,
    \end{align}
    where $\sigma_1, \ldots \sigma_p$ are the $p$ Gauss-Legendre integration nodes in $T$ and $l_k(s)=\prod_{j=1,j\neq k}^p \frac{(s-\sigma_j)}{\sigma_k-\sigma_j}$, the $k$th Lagrange interpolation polynomial. By multiplying \eqref{eq:starting_point} with $l_k$ as a test function and integrating over $T$ we get by the definition of the numerical scheme for $k=1,\ldots,p$ that
    \begin{align}
    \label{eq:test_fun}
        \int_T \Big(\partial_t y_\TT(t)-\partial_t y_{\widehat\TT}(t)\Big)l_k(t) dt=\int_T\Big(F(t,y_\TT(t))-F(t,y_{\widehat\TT}(t)) \Big)l_k(t) dt + \int_T R(t) l_k(t) dt.
    \end{align}
    Since the integrand on the left-hand side of \eqref{eq:test_fun} is a polynomial of order at most $2p-2$, the $p$ point Gauss-Legendre quadrature integrates it exactly, i.e., by denoting the corresponding integration weights by $w_1, \ldots, w_p$, we get
    \begin{align*}
        \int_T \Big(\partial_t y_\TT(t)-\partial_t y_{\widehat\TT}(t)\Big)l_k(t) dt=\sum_{i=1}^p w_i |T| \Big(\partial_t y_\TT(\sigma_i)-\partial_t y_{\widehat\TT}(\sigma_i)\Big)l_k(\sigma_i)=w_k |T| \Big(\partial_t y_\TT(\sigma_k)-\partial_t y_{\widehat\TT}(\sigma_k)\Big).
    \end{align*}
    Since $l_k$ is uniformly bounded in $T$ by a constant only depending on $p$, we get by taking norms in \eqref{eq:test_fun} that all $k=1, \ldots, p$ satisfy
    \begin{align*}
       |T| |\partial_t y_\TT(\sigma_k)-\partial_t y_{\widehat\TT}(\sigma_k)| \lesssim \int_T \Big|F(t,y_\TT(t))-F(t,y_{\widehat\TT}(t)) \Big|dt+ \int_T |R(t)|dt.
    \end{align*}
    Using the Lipschitz condition on $F$ shows
    \begin{align*}
       |T| |\partial_t y_\TT(\sigma_k)-\partial_t y_{\widehat\TT}(\sigma_k)| &\lesssim L_1 \int_T \Big|y_\TT(t)-y_{\widehat\TT}(t) \Big|dt+ \int_T |R(t)|dt\\
        &=L_1 \int_T \Big|y_{\TT}(t_i)-y_{\widehat\TT}(t_i)+ \int_{t_i}^t \sum_{k=1}^p \partial_t (y_\TT(\sigma_k)-y_{\widehat\TT}(\sigma_k)) l_k(s) ds \Big|dt + \int_T |R(t)|dt
    \end{align*}
    The boundedness of the Lagrange polynomials implies
    \begin{align*}
       |T| |\partial_t y_\TT(\sigma_k)-\partial_t y_{\widehat\TT}(\sigma_k)| &\lesssim |T|\Big(|y_{\TT}(t_i)-y_{\widehat\TT}(t_i)|+|T|\sum_{k=1}^p |\partial_t (y_\TT(\sigma_k)-y_{\widehat\TT}(\sigma_k))| \Big) + \int_T |R(t)|dt.
    \end{align*}
    Summing up over $k=1,\ldots,p$, we get
    \begin{align*}
      |T| \sum_{k=1}^p |\partial_t y_\TT(\sigma_k)-\partial_t y_{\widehat\TT}(\sigma_k)| &\leq C |T|\Big(|y_{\TT}(t_i)-y_{\widehat\TT}(t_i)|+|T|\sum_{k=1}^p |\partial_t (y_\TT(\sigma_k)-y_{\widehat\TT}(\sigma_k))| \Big) + \int_T |R(t)|dt
    \end{align*}
    for some constant $C>0$. This can be rearranged to 
    \begin{align}
    \label{eq:deriv_esti}
        |T| \sum_{k=1}^p |\partial_t y_\TT(\sigma_k)-\partial_t y_{\widehat\TT}(\sigma_k)|\leq C\frac{|T||y_{\TT}(t_i)-y_{\widehat\TT}(t_i)|+\int_T |R(t)|dt}{1-C|T|}
    \end{align}
    if $|T|$ is sufficiently small enough (this can be guaranteed by a sufficiently fine initial mesh $\TT_0$).
    Note that with~\eqref{eq:parametrize} this shows for any $t \in T$ that
    \begin{align}
    \label{eq:end_estimate}
        |y_\TT(t)-y_{\widehat\TT}(t)| \lesssim \frac{1+C|T|}{1-C|T|} \Big(|y_{\TT}(t_i)-y_{\widehat\TT}(t_i)|+\int_T |R(t)|dt\Big).
    \end{align}
    \emph{Step~2: }
    Let $[t_i,t_{i+1}]=T \in \TT \cap \widehat\TT$ be a non-refined element. Then we have by the definition of the numerical scheme that
    \begin{align*}
        \int_T \Big(\partial_t y_\TT(t)-\partial_t y_{\widehat\TT}(t)\Big)l_k(t) dt=\int_T\Big(F(t,y_\TT(t))-F(t,y_{\widehat\TT}(t)) \Big)l_k(t) dt.
    \end{align*}
    Analogously to Step~1, we get the estimates
    \begin{align}
        \label{deriv_esti_non_ref}
         |T| \sum_{k=1}^p |\partial_t y_\TT(\sigma_k)-\partial_t y_{\widehat\TT}(\sigma_k)|\leq C\frac{|T||y_{\TT}(t_i)-y_{\widehat\TT}(t_i)|}{1-C|T|}\quad\text{and}\\
         \label{eq:end_estimate_non_ref}
          |y_\TT(t)-y_{\widehat\TT}(t)| \lesssim \frac{1+C|T|}{1-C|T|} \Big(|y_{\TT}(t_i)-y_{\widehat\TT}(t_i)|\Big).
    \end{align}
    We prove by mathematical induction that for all $t \in [t_i,t_{i+1}]=T \in \widehat\TT$, there holds
    \begin{align}
    \label{eq:induction_hyp}
        |y_\TT(t)-y_{\widehat\TT}(t)| \lesssim \prod_{j=0}^i \frac{1+C|t_{j+1}-t_j|}{1-C|t_{j+1}-t_j|} \sum_{T'=[t'_j,t'_{j+1}] \in \widehat\TT \setminus \TT, t'_{j+1}\leq t_{i+1}} \int_{T'} |R(t)| dt.
    \end{align}
    The base case $i=0$ follows from $y_{\TT}(t_0)=y_{\widehat\TT}(t_0)$ and \eqref{eq:end_estimate} or \eqref{eq:end_estimate_non_ref} accordingly. For the inductive step we consider two cases: First suppose that $[t_{i+1},t_{i+2}] \in \widehat\TT \setminus \TT$. In this case \eqref{eq:end_estimate} implies that
    \begin{align*}
        |y_\TT(t)-y_{\widehat\TT}(t)| \lesssim \frac{1+C|t_{i+2}-t_{i+1}|}{1-C|t_{i+2}-t_{i+1}|} \Big(|y_{\TT}(t_{i+1})-y_{\widehat\TT}(t_{i+1})|+\int_{t_{i+1}}^{t_{i+2}} |R(t)|dt\Big).
    \end{align*}
    Together with the induction hypothesis \eqref{eq:induction_hyp} this shows
    \begin{align*}
         |y_\TT(t)&-y_{\widehat\TT}(t)| \\&\lesssim \frac{1+C|t_{i+2}-t_{i+1}|}{1-C|t_{i+2}-t_{i+1}|} \Bigg( \prod_{j=0}^i \frac{1+C|t_{j+1}-t_j|}{1-C|t_{j+1}-t_j|} \sum_{T'=[t'_j,t'_{j+1}] \in \widehat\TT \setminus \TT, t'_{j+1}\leq t_{i+1}} \int_{T'} |R(t)| dt + \int_{t_{i+1}}^{t_{i+2}} |R(t)|dt\Bigg)\\
         &\lesssim \prod_{j=0}^{i+1} \frac{1+C|t_{j+1}-t_j|}{1-C|t_{j+1}-t_j|} \sum_{T'=[t'_j,t'_{j+1}] \in \widehat\TT \setminus \TT, t'_{j+1}\leq t_{i+2}} \int_{T'} |R(t)| dt.
    \end{align*}
    The case $[t_{i+1},t_{i+2}] \in \TT \cap \widehat\TT$ follows analogously using \eqref{eq:end_estimate_non_ref} instead of \eqref{eq:end_estimate}. This concludes the induction and proves \eqref{eq:induction_hyp}. Note that the denominators in the products are positive for sufficiently fine initial discretizations. From this and Lemma \ref{lem:prod_bound} below, we deduce immediately that
    \begin{align*}
        \norm{y_{\TT}-y_{\widehat\TT}}{L^2(T)}^2 \lesssim  |T| \sum_{T' \in \widehat\TT \setminus \TT}\int_{T'} |R(t)|^2 dt.
    \end{align*}
  Summing over all $T$ yields
    \begin{align}
        \label{eq:l2_esti_gen}
        \norm{y_{\TT}-y_{\widehat\TT}}{L^2([t_0,\tend])}^2 \lesssim \sum_{T' \in \widehat\TT \setminus \TT}\int_{T'} |R(t)|^2 dt.
    \end{align}
    \emph{Step~3:} It remains to prove the same estimate in the $H^1$-norm. To that end we have for $T \in \TT \cap \widehat\TT$ that
    \begin{align*}
        \norm{\partial_t y_\TT- \partial_t y_{\widehat\TT}}{L^2(T)}^2 =\int_T\Big(F(t,y_\TT(t))-F(t,y_{\widehat\TT}(t))\Big) \cdot \Big(\partial_t y_\TT(t)- \partial_t y_{\widehat\TT}(t) \Big)dt
    \end{align*}
    by the definition of the numerical scheme. Cauchy-Schwarz and Lipschitz continuity of $F$ give
    \begin{align}
        \label{eq:h1_non_ref}
        \norm{\partial_t y_\TT- \partial_t y_{\widehat\TT}}{L^2(T)} \leq L_1 \norm{y_{\TT}-y_{\widehat\TT}}{L^2(T)}.
    \end{align}
    For a refined element $T \in \widehat\TT \setminus \TT$, we have that
    \begin{align*}
        \norm{&\partial_t y_\TT- \partial_t y_{\widehat\TT}}{L^2(T)}^2\\
        &=\int_T \Big(F(t,y_\TT(t))-F(t,y_{\widehat\TT}(t))\Big) \cdot \Big(\partial_t y_\TT(t)- \partial_t y_{\widehat\TT}(t) \Big)dt+ \int_T R(t) \cdot \Big(\partial_t y_\TT(t)- \partial_t y_{\widehat\TT}(t) \Big)dt.
    \end{align*}
    Lipschitz continuity and Cauchy-Schwartz again give
    \begin{align}
        \label{eq:h1_ref}
        \norm{\partial_t y_\TT- \partial_t y_{\widehat\TT}}{L^2(T)} \leq L_1 \Big(\norm{y_{\TT}-y_{\widehat\TT}}{L^2(T)}+\norm{R}{L^2(T)} \Big).
    \end{align}
    Summing over all $T$ using \eqref{eq:h1_non_ref} and \eqref{eq:h1_ref} yields that
    \begin{align*}
        \norm{\partial_t y_\TT- \partial_t y_{\widehat\TT}}{L^2([t_0,\tend])}^2 \lesssim \norm{y_{\TT}-y_{\widehat\TT}}{L^2([t_0,\tend])}^2+ \sum_{T' \in \widehat\TT \setminus \TT}\int_{T'} |R(t)|^2 dt.
    \end{align*}
    Using the $L^2$ estimate \eqref{eq:l2_esti_gen} thus concludes the proof.
    \end{proof}
On the way to linear convergence~\eqref{eq:linconv} of the estimator sequence we finally need~(A3) \emph{quasi-orthogonality}.
However, to deal with nonlinear problems, we need to introduce a linearization of \eqref{eq:ode}. With the solution $y$ of~\eqref{eq:ode}, we state the linearized problem.
\begin{align}
\label{eq:ode_lin_gen}
    \begin{split}
        \partial_t \widetilde{y}(t)&=F(t,y(t))+\nabla_yF(t,y(t))\Big( \widetilde{y}(t)-y(t)\Big) \quad \text{for} \quad t \in [t_0, \tend] \\
        \widetilde{y}(t_0)&=y_0.
    \end{split}
\end{align}
Note that since the right-hand side is Lipschitz continuous by definition, $\widetilde{y}=y$ is the unique solution of \eqref{eq:ode_lin_gen}. We discretize \eqref{eq:ode_lin_gen} on the same time mesh $\TT$ and with the same method we used for the non-linear problem, i.e., find an approximate solution $\widetilde{y}_\TT \in S^p(\TT)$ such that
\begin{align}
\label{eq:lin_disc_gen}
    \begin{split}
        \int_{t_0}^{\tend}\Big(\partial_t \widetilde{y}_\TT(t)-F(t,y(t))+\nabla_yF(t,y(t))\Big( \widetilde{y}_\TT(t)-y(t)\Big) \Big)\cdot v(t) dt &=0,\\
    \widetilde{y}_{\TT}(t_0)&=y_0,
    \end{split}
\end{align}
for all $v \in \PP^{p-1}(\TT)$. Note that this auxiliary approximation serves theoretical purposes only and need not be computed. We are ready to state a lemma which will bound the difference between $y_\TT$ and $\widetilde{y}_\TT$.
\begin{lemma}
    \label{lem:lin_min_nonlin_gen}
     Let the right-hand side $F$ from~\eqref{eq:ode} satisfy \eqref{eq:Lip}--\eqref{eq:JacLip} and the let initial discretization $\TT_0$ be sufficiently fine. Then, we have for solutions $y_\TT$ of $\eqref{eq:gen_disc}$ and $\widetilde{y}_\TT$ of \eqref{eq:lin_disc_gen}
    \begin{align*}
        \norm{\widetilde{y}_\TT-y_\TT}{H^1([t_0,\tend])} \leq C \norm{y-y_\TT}{H^1([t_0,\tend])}^2,
    \end{align*}
    where the constant $C>0$ is independent of $\TT$.
\end{lemma}
\begin{proof}
    Define $w=\widetilde{y}_\TT-y_\TT$ and $a(t)=F(t,y(t))-F(t,y_\TT(t))+\nabla_y F(t,y(t))\Big( y_\TT(t)-y(t)\Big)$. Note that due to the assumptions on $F$, Taylor expansion shows that $a(t)=\mathcal{O}(|y_\TT(t)-y(t)|^2)$, where the implied constants do not depend on $t$. Let $[t_i, t_{i+1}]=T \in \TT$.  For every $k= 1.\ldots, p$ it holds by the definition of the numerical scheme that
    \begin{align}
    \label{eq:diffeq}
        \int_T \partial_t w(t)l_k(t)dt=\int_T \Big(a(t)+\nabla_yF(t,y(t))w(t) \Big)l_k(t)dt,
    \end{align}
    which, as in the proof of Lemma \ref{lem:gen_reldisc}, shows together with the boundedness of $\nabla_yF(t,y(t))$ that for all $t \in T$, we have
    \begin{align*}
        |w(t)| \lesssim \frac{1+C|T|}{1-C|T|}\Big(|w(t_i)|+\int_T|a(t)|dt \Big).
    \end{align*}
    By a similar inductive argument as in the proof of Lemma~\ref{lem:gen_reldisc} and using that $w(t_0)=0$, we show for $t \in T=[t_i,t_{i+1}]$ that
    \begin{align*}
        |w(t)| \lesssim \prod_{j=0}^{i}\frac{1+C|t_{j+1}-t_j|}{1-C|t_{j+1}-t_j|} \int_{t_0}^{t_{i+1}}|a(t)|dt,
    \end{align*}
    which, together with the boundedness of the product by Lemma \ref{lem:prod_bound} below, implies the bound
    \begin{align*}
        \norm{w}{L^2([t_0,\tend])} \lesssim \int_{t_0}^{\tend}|a(t)|dt =\mathcal{O}\Big(\norm{y_\TT-y}{L^2([t_0,\tend])}^2\Big).
    \end{align*}
    By choosing $\partial_t w$ as a test function in \eqref{eq:diffeq} and with the boundedness of $\nabla_yF(t,y(t))$, we infer
    \begin{align*}
        \norm{\partial_t w}{L^2([t_0,\tend])} \lesssim \norm{a}{L^2([t_0,\tend])}+\norm{w}{L^2([t_0,\tend])}.
    \end{align*}
    Inserting the $L^2$ estimate, we conclude the proof.
\end{proof}
For the next result, we use an equivalent weak formulation of \eqref{eq:lin_disc_gen} as introduced in \cite{generalqo}. Find $\widetilde{y}_\TT \in S^p(\TT)$ such that $\widetilde{y}_\TT(t_0)=0$ and all $v \in \PP^{p-1}(\TT)$ satisfy
\begin{align}
\label{eq:weak_lin_gen}
    a(\widetilde{y}_\TT,v)= \int_{t_0}^{\tend}\Big(F(t,y(t))- \nabla_y F(t,y(t))(y(t)-y_0) \Big) \cdot v(t) dt,
\end{align}
where 
\begin{align*}
    a(\widetilde{y}_\TT,v)=\int_{t_0}^{\tend}\Bigg(\partial_t  \widetilde{y}_\TT(t)-\nabla_yF(t,y(t))\widetilde{y}_\TT(t) \Bigg) \cdot v(t) dt
\end{align*}
for $\widetilde{y}_\TT \in \mathcal{X}_\TT=\{u \in S^p(\TT): u(t_0)=0\}$ and $v \in  \mathcal{Y}_\TT = \PP^{p-1}(\TT)$.
Note that $\widetilde{y}_\TT(t)+y_0$ will solve \eqref{eq:lin_disc_gen}. The weak formulation $\eqref{eq:weak_lin_gen}$ fits into the abstract framework of \cite{generalqo}.

 We will equip $\mathcal{X}_\TT$ and $\mathcal{Y}_\TT$ with the norms $\norm{\cdot}{H^1([t_0,\tend])}$ and $\norm{\cdot}{L^2([t_0,\tend])}$,  respectively. Continuity of $a(\cdot,\cdot)$ follows trivially. Moreover, we have inf-sup stability as shown in the following lemma.
\begin{lemma}
    \label{lem:infsup_gen}
     If the right-hand side $F$ from~\eqref{eq:ode} satisfies \eqref{eq:Lip}--\eqref{eq:JacLip} and the initial discretization $\TT_0$ is sufficiently fine, $a(\cdot,\cdot)$ is uniformly inf-sup stable in the sense that for all $\TT$ which can be obtained by some refinement of $\TT_0$,
    \begin{align}
         \inf_{u \in \mathcal{X}_{\TT}} \sup_{v \in \mathcal{Y}_{\TT}} \frac{a(u,v)}{\norm{u}{\mathcal{X}_{\TT}} \norm{v}{\mathcal{Y}_{\TT}}} \geq \gamma >0.
    \end{align}
    The constant $\gamma$ does not depend on $\TT$.
\end{lemma}
\begin{proof}
    Let $\TT\in\T$ be an arbitrary refinement of $\TT_0$ and let $\widetilde{y}_{\TT} \in \mathcal{X}_{\TT}$. Choose $v(t)=\Pi_{p-1} \Big( \partial_t \widetilde{y}_{\TT}(t)- \nabla_yF(t,y(t))\widetilde{y}_{\TT}(t)  \Big)$,  where $\Pi_{p-1}: L^2([t_0,\tend],\R^d) \to \PP^{p-1}(\TT)$ denotes the $L^2$-orthogonal projection onto piecewise polynomials of degree $p-1$.
    Substituting into the bilinear form yields
     \begin{align}
     \label{eq:bilin_gen}
     \begin{split}
         a(\widetilde{y}_{\TT},v)&=\int_{t_0}^{\tend}\Big(\partial_t  \widetilde{y}_{\TT}(t)-\nabla_yF(t,y(t))\widetilde{y}_{\TT}(t) \Big) \cdot v(t) dt\\
    &=\int_{t_0}^{\tend}\Big(\partial_t  \widetilde{y}_{\TT}(t)-\nabla_yF(t,y(t))\widetilde{y}_{\TT}(t) \Big) \cdot \Big( \partial_t \widetilde{y}_{\TT}(t)- \nabla_yF(t,y(t))\widetilde{y}_{\TT}(t)\Big) dt \\
    &\qquad +\int_{t_0}^{\tend}\Big(\partial_t  \widetilde{y}_{\TT}(t)-\nabla_yF(t,y(t))\widetilde{y}_{\TT}(t) \Big) \cdot \Big(v(t)- \Big(\partial_t \widetilde{y}_{\TT}(t)- \nabla_yF(t,y(t))\widetilde{y}_{\TT}(t)\Big)\Big) dt\\
    &:= {\rm I} + {\rm II}.
     \end{split}
     \end{align}
     To analyze the first term, we notice
     \begin{align}
     \label{eq:ode_self_gen}
          \partial_t \widetilde{y}_{\TT}(t)=\nabla_yF(t,y(t))\widetilde{y}_{\TT}(t)+ \partial_t \widetilde{y}_{\TT}(t)- \nabla_yF(t,y(t))\widetilde{y}_{\TT}(t).
     \end{align}
      Integrating from $t_0$ to $t$ yields, with a constant only depending on an upper bound for $\nabla_yF(t,y(t))$, that
      \begin{align*}
          |\widetilde{y}_{\TT}(t)| \leq C_F \int_{t_0}^{t}|\widetilde{y}_{\TT}(s)|ds + \int_{t_0}^{t} |\partial_t \widetilde{y}_{\TT}(s)- \nabla_yF(s,y(s))\widetilde{y}_{\TT}(s)| ds.
      \end{align*}
      Gronwall's lemma gives
      \begin{align}
        \label{eq:esti_l2_gen}
          |\widetilde{y}_{\TT}(t)| \lesssim \int_{t_0}^{t} |\partial_t \widetilde{y}_{\TT}(s)- \nabla_yF(s,y(s))\widetilde{y}_{\TT}(s)| ds.
      \end{align}
      Integrating the square of \eqref{eq:ode_self_gen} shows
       \begin{align}
     \label{eq:bound1_gen}
         \norm{\partial_t \widetilde{y}_{\TT}}{L^2([t_0,\tend])}^2 \lesssim \Big( \int_{t_0}^{\tend} \Big|\nabla_yF(t,y(t))\widetilde{y}_{\TT}(t)\Big|^2 dt+\int_{t_0}^{\tend} \Big| \partial_t \widetilde{y}_{\TT}(t)- \nabla_yF(t,y(t))\widetilde{y}_{\TT}(t)\Big|^2 dt \Big).
     \end{align}
     Since $\nabla_yF(t,y(t))$ is bounded,~\eqref{eq:esti_l2_gen} implies
     \begin{align*}
         \norm{\partial_t \widetilde{y}_{\TT}}{L^2([t_0,\tend])}^2 \lesssim \norm{\partial_t \widetilde{y}_{\TT}- \nabla_yF(\cdot,y(\cdot))\widetilde{y}_{\TT}}{L^2([t_0,\tend])}^2 = {\rm I}.
         \end{align*}
         This concludes the estimate for the first term on the right-hand side of \eqref{eq:bilin_gen}.
     For the second term, we need to bound $\Big|v(t)- \Big(\partial_t \widetilde{y}_{\TT}(t)- \nabla_yF(t,y(t))\widetilde{y}_{\TT}(t)\Big)\Big|$. Since $\partial_t \widetilde{y}_{\TT} \in \PP^{p-1}(\TT)$, we have $\Pi_{p-1}\partial_t \widetilde{y}_{\TT}=\partial_t \widetilde{y}_{\TT}$. Moreover, with the Lagrange polynomials already used in the proof of Lemma~\ref{lem:gen_reldisc}, we have for $t \in \TT$ that
     \begin{align}
     \begin{split}
         \int_T\Big|v(t)- \Big(&\partial_t \widetilde{y}_{\TT}(t)- \nabla_yF(t,y(t))\widetilde{y}_{\TT}(t)\Big)\Big|^2dt \\
         &\leq \int_T\Big|\sum_{k=1}^pl_k(t)\nabla_yF(\sigma_k,y(\sigma_k))y_{\TT}(\sigma_k)-\nabla_yF(t,y(t))y_{\TT}(t)\Big|^2 dt\\
         &\lesssim \sum_{k=1}^p \int_T  |l_k(t) y_{\TT}(\sigma_k)|^2\Big|\nabla_yF(\sigma_k,y(\sigma_k))- \nabla_yF(t,y(t))\Big|^2 dt\\
         &\qquad+ \int_T  |\nabla_yF(t,y(t))|^2\Big|\sum_{k=1}^pl_k(t) y_{\TT}(\sigma_k)-y_{\TT}(t) \Big|^2 dt. 
         \end{split}
     \end{align}
     By using interpolation estimates for polynomials and the Lipschitz continuity of $\nabla F$ and differentiability of $y$ this yields
     \begin{align*}
         \int_T\Big|v(t)-& \Big(\partial_t \widetilde{y}_{\TT}(t)- \nabla_yF(t,y(t))\widetilde{y}_{\TT}(t)\Big)\Big|^2dt \\
         &\lesssim \sum_{k=1}^p \int_T  |l_k(t) y_{\TT}(\sigma_k)|^2\|T|^2 \Big(||\partial_ty||_{L^\infty((t_0,\tend))}+1\Big)^2 dt+ |T|^2\norm{\partial_t y_{\TT}}{L^2(T)}^2.
     \end{align*}
     Norm equivalence on the finite dimensional space $\PP^p(T)$ together with a scaling argument conclude
     \begin{align*}
         \norm{v(t)- \Big(\partial_t \widetilde{y}_{\TT}(t)- \nabla_yF(t,y(t))\widetilde{y}_{\TT}(t)\Big)}{L^2(T)} \lesssim |T| \Big(\norm{y_{\TT}}{L^2(T)}+\norm{\partial_t y_{\TT}}{L^2(T)} \Big).
     \end{align*}
     Summing all $T$ and Poincaré inequality combined with Cauchy-Schwartz yields an upper bound for the second term in \eqref{eq:bilin_gen} of the form
    \begin{align}
     \label{eq:bound2_gen}
         {\rm II}\leq C_F \max_{T \in \TT} |T| \norm{\partial_t \widetilde{y}_{\TT}}{L^2([t_0,t_{\rm end}])},
     \end{align}
     where $C_F$ is independent of the mesh.
     Therefore, we obtain with \eqref{eq:bound1_gen} and \eqref{eq:bound2_gen} the bound
     \begin{align*}
          a(\widetilde{y}_{\TT},v) \gtrsim  (1-C_F\max_{T \in \TT_0} |T|) \norm{\partial_t \widetilde{y}_{\TT}}{L^2([t_0,\tend])}^2,
     \end{align*}
     where the term in the brackets is positive if the initial discretization is sufficiently fine. With $\norm{v}{L^2([t_0,\tend])} \lesssim \norm{\partial_t \widetilde{y}_{\TT}}{L^2([t_0,\tend])}$, we conclude the proof.
\end{proof}
Let us investigate the boundedness of the product used in the preceding lemmas.
\begin{lemma}
\label{lem:prod_bound}
    Let $\theta=1/2, \lambda>0$ and $\{\tau_n\} \subset \R^+$ such that $\sum_n \tau_n=S< \infty$ and $1-\theta \lambda \tau_n>1-\theta \lambda c>0, c >0$, for all $n \in \N$. Then
    \begin{align*}
        \prod_n \frac{1+ \theta \lambda \tau_n}{1-\theta \lambda \tau_n} \leq \exp \Big(S \theta \lambda-\frac{S}{c}\log(1-\theta \lambda c) \Big).
    \end{align*}
\end{lemma}
\begin{proof}
    For the numerator it holds that
    \begin{align*}
        1+\theta \lambda \tau_n \leq \exp(\theta \lambda \tau_n),
    \end{align*}
    which yields
    \begin{align*}
        \prod_n (1+\theta \lambda \tau_n) \leq \exp (S \theta \lambda).
    \end{align*}
    For the denominator we have
    \begin{align*}
        \log(1-\theta \lambda \tau_n) \geq \frac{\tau_n}{c}\log(1-\theta \lambda c)
    \end{align*}
    because of the concavity of $\log(1-\cdot)$. This yields
    \begin{align*}
        \log\Bigg(\prod_n (1-\theta \lambda \tau_n) \Bigg) \geq \frac{S}{c}\log(1-\theta \lambda c).
    \end{align*}
    Combining both estimates concludes the proof.
\end{proof}
We are now ready to prove quasi-orthogonality (A3). To that end, we consider a sequence of sequentially refined meshes $\TT_{\ell}$ for $\ell \geq 0$ and  the corresponding solutions $y_{\ell}\in \SS^p(\TT_\ell)$ of \eqref{eq:gen_disc} as well as $\widetilde{y}_{\ell}\in \SS^p(\TT_\ell)$ of \eqref{eq:lin_disc_gen}. 
    \begin{lemma}
        \label{lem:quasi_ortho}
        Let the right-hand side $F$ from~\eqref{eq:ode} satisfy \eqref{eq:Lip}--\eqref{eq:JacLip} and let the initial discretization $\TT_0$ be sufficiently fine. Then, $y_{\ell}$ satisfies general quasi-orthogonality in the sense that there exist a function $C:\N \to \R$ and $\eps>0$ small such that for every $\ell,N \in \N$ we have
        \begin{align*}
            \sum_{k=\ell}^{\ell+N} \norm{y_{k+1}-y_k}{H^1([t_0,\tend])}^2-\eps \eta_{\TT_k}^2 \leq C(N) \norm{y-y_{\ell}}{H^1([t_0,\tend])}^2,
        \end{align*}
        where $C(N)=\mathcal{O}(N^{1-\delta})$ for some $\delta>0$.
    \end{lemma}
\begin{proof}
    Let $\ell,N \in \N$, we have
    \begin{align*}
        \sum_{k=\ell}^{\ell+N}& \norm{y_{k+1}-y_k}{H^1([t_0,\tend])}^2\\
        & \leq  3\sum_{k=\ell}^{\ell+N} \Big(\norm{y_{k+1}-\widetilde{y}_{k+1}}{H^1([t_0,\tend])}^2+\norm{y_k-\widetilde{y}_k}{H^1([t_0,\tend])}^2+\norm{\widetilde{y}_{k+1}-\widetilde{y}_k}{H^1([t_0,\tend])}^2\Big).
    \end{align*}
    The result~\cite[Lemma 9]{generalqo} shows that Lemma \ref{lem:infsup_gen} implies general quasi-orthogonality with $C(N)=\mathcal{O}(N^{1-\delta})$, i.e.,
    \begin{align*}
         \sum_{k=\ell}^{\ell+N} \norm{y_{k+1}-y_k}{H^1([t_0,\tend])}^2\leq C(N) \norm{y-y_{\ell}}{H^1([t_0,\tend])}^2
    \end{align*}
    for all $\ell,N\in\N$, if the initial mesh is fine enough. Lemma \ref{lem:lin_min_nonlin_gen} implies that with a constant $C_{ \rm orth}$ independent of the mesh, there holds
    \begin{align*}
        \sum_{k=\ell}^{\ell+N} &\norm{y_{k+1}-y_k}{H^1([t_0,\tend])}^2 \\
        &\leq C_{\rm orth} \sum_{k=\ell}^{\ell+N} \Big(\norm{y-y_k}{H^1([t_0,\tend])}^4+ \norm{y-y_{k+1}}{H^1([t_0,\tend])}^4 \Big) +C(N)\norm{y-\widetilde{y}_{\ell}}{H^1([t_0,\tend])}^2.
    \end{align*}
    Using reliability from Lemma \ref{lem:gen_reliability} together with stability and reduction (Lemmas \ref{lem:gen_stability}--\ref{lem:gen_reduction}) and choosing the initial mesh sufficiently fine such that $\norm{y-y_{\ell}}{H^1([t_0,\tend])} \leq \eps$ (this is possible because of the convergence of our method on uniform meshes, see Lemma \ref{lem:unifconv}), we get
    \begin{align*}
         \sum_{k=\ell}^{\ell+N} \norm{y_{k+1}-y_k}{H^1([t_0,\tend])}^2-\eps \eta_{\TT_k}^2 &\leq C(N)\norm{y-\widetilde{y}_{\ell}}{H^1([t_0,\tend])}^2.
    \end{align*}
    Another application of Lemma \ref{lem:lin_min_nonlin_gen} to the right-hand side yields the result.
\end{proof}
Now we are ready to proof the main result of this section.
\begin{proof}[Proof of Theorem \ref{thm:optim}]
    Suppose $\TT_0$ is sufficiently fine such that all the Lemmas of this section are applicable. The results \cite[Lemma 3.5 \& 4.7]{Carstensen_2014} show that we can deduce \textit{estimator reduction} and \textit{quasi-monotonicity} from Lemmas \ref{lem:gen_reliability}, \ref{lem:gen_stability}--\ref{lem:gen_reduction}. Therefore, the requirements for \cite[Lemmas 5\& 6]{generalqo} are met. Together with \cite[Remark 7]{generalqo} and Lemma \ref{lem:quasi_ortho}, this shows for the estimator sequence from Algorithm \ref{alg:adaptive} that~\eqref{eq:linconv} holds for uniform constants $0 < q <1$ and $C_{\rm lin}>0$. With this we can use the results \cite[Prop. 4.12, 4.14, 4.15]{Carstensen_2014} together with Lemma \ref{lem:gen_reldisc} and obtain $0<\theta_\star<1$ such that Algorithm \ref{alg:adaptive} is rate-optimal~\eqref{eq:optimal}.

To show~\eqref{eq:runtime}, we introduce the notation ${\rm time}_\ell$ for the total runtime time in~\eqref{eq:runtime}. With~\eqref{eq:optimal} and~\eqref{eq:linconv}, ${\rm time}_\ell$  can be estimated by
\begin{align*}
    \text{time}_\ell \leq C_{\rm time}  \sum_{j=0}^\ell \#\TT_\ell
    &\leq C_{\rm time}C_{\rm best}^{1/s}C_{\rm opt}^{1/s} \sum_{j=0}^\ell \eta_{\TT_j}^{-1/s} \leq C_{\rm time}C_{\rm best}^{1/s}C_{\rm opt}^{1/s}C_{\rm lin}^{1/s}\eta_{\TT_\ell}^{-1/s} \sum_{j=0}^\ell q^{(\ell-j)/s}\\
    &\leq C_{\rm time}\frac{C_{\rm best}^{1/s}C_{\rm opt}^{1/s}C_{\rm lin}^{1/s}}{1-q^{1/s}}\eta_{\TT_\ell}^{-1/s}.
\end{align*}
This concludes the proof.

\end{proof}

\printbibliography 
\end{document}

%% file: mymacros.tex
\def\R{{\mathbb R}}

\def\N{{\mathbb N}}

\def\PP{{\mathcal P}}

\def\SS{{\mathcal S}}

\def\TT{{\mathcal T}}

\def\XX{{\mathcal X}}

\def\tend{{t_{\rm end}}}

\def\b1{{\boldsymbol{1}}}

\def\norm#1#2{\|#1\|_{#2}}

\def\set#1#2{\big\{#1\,:\,#2\big\}}

\def\eps{\varepsilon}
\def\normL2#1#2{\|#1\|_{L^2(#2)}}

\newcounter{constantsnumber}
\def\namec#1#2{%
 \ifthenelse{\equal{#1}{rel}}{C_{\rm rel}}{%
  \ifthenelse{\equal{#1}{mesh}}{C_{\rm mesh}}{%
  \ifthenelse{\equal{#1}{sz}}{C_{\rm sz}}{%
  \ifthenelse{\equal{#1}{dislocrel}}{C_{\rm dlr}}{%
  \ifthenelse{\equal{#1}{eff}}{C_{\rm eff}}{%
  \ifthenelse{\equal{#1}{main}}{C_{\rm V}}{%
  \ifthenelse{\equal{#1}{opt}}{C_{\rm opt}}{%
  \ifthenelse{\equal{#1}{normequiv}}{C_{\rm norm}}{%
  \ifthenelse{\equal{#1}{reliable}}{C_{\rm rel}}{%
  \ifthenelse{\equal{#1}{efficient}}{C_{\rm eff}}{%
  \ifthenelse{\equal{#1}{dlr}}{C_{\rm dlr}}{%
  \ifthenelse{\equal{#1}{stable}}{C_{\rm stab}}{%
  \ifthenelse{\equal{#1}{reduction}}{C_{\rm red}}{%
   \ifthenelse{\equal{#1}{unibound}}{C_{\rm hot}}{%
    \ifthenelse{\equal{#1}{hotConst}}{C_{\rm hot}}{%
   \ifthenelse{\equal{#1}{inverseK}}{C_{\rm K}}{%
  \ifthenelse{\equal{#1}{refined}}{C_{\rm ref}}{%
  \ifthenelse{\equal{#1}{estconv}}{C_{\rm est}}{%
  \ifthenelse{\equal{#1}{optimal}}{C_{\rm opt}}{%
  \ifthenelse{\equal{#1}{qo}}{C_{\rm qo}}{%
  \ifthenelse{\equal{#1}{mon}}{C_{\rm mon}}{%
  \ifthenelse{\equal{#1}{cea}}{C_{\mbox{\scriptsize C\'ea}}}{%
  \ifthenelse{\equal{#2}{newcounter}}{\refstepcounter{constantsnumber}\label{const#1}}{}C_{\ref{const#1}}}%
}}}}}}}}}}}}}}}}}}}}}}

\newcounter{contractionnumber}
\def\nameq#1#2{%
  \ifthenelse{\equal{#1}{reduction}}{q_{\rm red}}{%
  \ifthenelse{\equal{#1}{estconv}}{q_{\rm est}}{%
  \ifthenelse{\equal{#1}{cea}}{q_{\mbox{\scriptsize C\'ea}}}{%
  \ifthenelse{\equal{#2}{newcounter}}{\refstepcounter{contractionnumber}\label{contraction#1}}{}q_{\ref{contraction#1}}}%
}}}

\def\namer#1#2{%
  \ifthenelse{\equal{#1}{reduction}}{\rho_{\rm red}}{%
  \ifthenelse{\equal{#1}{estconv}}{\rho_{\rm est}}{%
  \ifthenelse{\equal{#1}{cea}}{\rho_{\mbox{\scriptsize C\'ea}}}{%
  \ifthenelse{\equal{#1}{qo}}{\rho_{\mbox{\scriptsize qo}}}{%
  \ifthenelse{\equal{#2}{newcounter}}{\refstepcounter{contractionnumber}\label{contraction#1}}{}\rho_{\ref{contraction#1}}}%
}}}}

\newtheorem{theorem}{Theorem}
\newtheorem{remark}[theorem]{Remark}

\newtheorem{lemma}[theorem]{Lemma}

\def\T{\mathbb T}